\documentclass[12pt]{amsart}

\usepackage{amsmath, amssymb, amsfonts, amsbsy, amsthm, latexsym, color}
\usepackage{mathdots}
\usepackage{graphicx}
\usepackage{enumerate}

\usepackage{algpseudocode}
\usepackage{tikz,tikz-cd}
\usetikzlibrary{decorations.pathreplacing}
\usetikzlibrary{calc}
\usepackage{hyperref}

\usepackage[utf8]{inputenc}
\everymath{\displaystyle}
\newtheorem{thm}{Theorem}[section]
\newtheorem{proposition}[thm]{Proposition}
\newtheorem{lemma}[thm]{Lemma}
\newtheorem{cor}[thm]{Corollary}

\theoremstyle{definition}
\newtheorem{dfn}[thm]{Definition}
\newtheorem{ex}[thm]{Example}

\newtheorem{remark}[thm]{Remark}
\newtheorem*{remark*}{Remark}

\newcommand{\C}{\mathbb{C}}
\newcommand{\CC}{\mathbb{C}}
\newcommand{\ZZ}{\mathbb{Z}}

\newcommand{\Fc}{\mathcal{F}}

\newcommand{\Tc}{\mathcal{T}}

\newcommand{\vphi}{{\varphi}}

\newcommand{\be}{\begin{equation}}
\newcommand{\ee}{\end{equation}}

\newcommand{\rank}{\text{rank}}
\newcommand{\mon}{{\text{Mon}}}
\newcommand{\diag}{{\rm{diag}}}

\newcommand{\CP}{\mathbb{CP}}
\newcommand{\FF}{\mathbb F}
\newcommand{\PP}{\mathbb P}
\newcommand{\leg}[2]{\left[\tfrac{#1}{#2}\right]}

\renewcommand{\i}{\mathrm{i}}

\DeclareMathOperator{\cir}{\mathrm{circ}}
\DeclareMathOperator{\negc}{\mathrm{neg}}

\newcommand{\etf}[2]{{\text{ETF}(#1,#2)}}
\newcommand{\etfn}{\etf{2n}{n}}
\newcommand{\tf}[2]{{\text{TF}(#1,#2)}}
\newcommand{\conf}[2]{{\text{CONF}(#1,#2)}}
\newcommand{\bc}{{bi-circulant}}
\newcommand{\bnc}{{bi-negacirculant}}
\DeclareMathOperator{\Aut}{\mathrm{Aut}}

\numberwithin{equation}{section}
\begin{document}

\title[Dihedral ETFS]{Abelian and Dihedral equiangular tight frames of redundancy $2$}

\author[R. Ben-Av]{Radel Ben-Av}
\address{Department of Computer Science, Holon Institute of Technology}
\email{benavr@hit.ac.il}
\author[X. Chen]{Xuemei Chen}
\address{Department of Mathematics and Statistics, University of North Carolina Wilmington}
\email{chenxuemei@uncw.edu}
\author[A. Goldberger]{Assaf Goldberger}
\address{School of Mathematical Sciences, Tel Aviv University}
\email{assafg@tauex.tau.ac.il}
\author[K. A. Okoudjou]{Kasso A. Okoudjou}
\address{Department of Mathematics, Tufts University, 177 College Avenue, Medford, MA 02155, USA}
\email{kasso.okoudjou@tufts.edu}

\date{\today}

\subjclass[2000]{Primary 42C15 Secondary 52C30}
\keywords{Equiangular tight frames, group frames, Paley ETF, diherdal symmetry}

\begin{abstract}
This paper studies group frames ($G$-frames) where the unitary group representation can be projective. 
When the group is abelian, for most combinations $N, n$, we show that $\etf{N}{n}$ can only exist for genuinely projective group representations. In particular, cyclic-group frames for such parameters do not exist. We also give a characterization of all dihedral tight frames and dihedral $\etfn$, using which, we conclude that regular dihedral $\etfn$ must be genuinely projective. Following that, we give a characterization of regular dihedral $\etfn$ in terms of certain structured skew Hadamard matrices. We then show that Paley $\etfn$ and its doubling are both of this type. Finally, we classify all regular dihedral $\etfn$ for $n\le 22$ up to switching equivalence.
\end{abstract}

\maketitle

\section{Introduction}\label{sec1}

\subsection{Background and notations}\label{subsec1.1}

An $N$-point configuration $\Phi=\{\vphi_j\}_{j=0}^{N-1}\subset\C^n$ is an {\bf ordered multiset} of nonzero vectors. We write $[\Phi]\in \CC^{n\times N}$ for the matrix whose columns are the vectors $\varphi_j$ in the prescribed order. By abuse of notation, most of the time we will just write $\Phi$ instead of $[\Phi]$. Given a configuration $\Phi$, if another configuration $\Psi$ is obtained by applying a unitary transformation to the left of $\Phi$, reindexing $\Phi$, or changing the phase of any of the vectors, or composition of any of the above operations, then $\Phi$ and $\Psi$ are said to be switching equivalent. This is an equivalence relation, and all configurations in this equivalence class should be considered the same. See Section \ref{sec:switching} for more details.

Our interest in $N$-point configurations stems  from their approximation property which is best described by the notion of a finite frame. $\Phi=\{\vphi_j\}_{j=0}^{N-1}\subset\C^n$ is called a \emph{frame} if it spans $\C^n$, or equivalently if  there exist $0<A\leq B<\infty$ such that 
\begin{equation}
A\|x\|^2\leq\sum_{j=0}^{N-1}|\langle x, \vphi_j\rangle|^2\leq B\|x\|^2\qquad \forall\ x\in\C^n,
\end{equation}
where $\|\cdot\|$ is the $\ell_2$ norm.  The frame $\Phi=\{\vphi_j\}_{j=0}^{N-1}$ is said to be \emph{tight} if $A=B$. 
Furthermore, a tight frame $\Phi=\{\varphi_j\}_{j=0}^{N-1}$ is called an \emph{equiangular tight frame (ETF)} if  $|\langle \vphi_i/\|\vphi_i\|, \vphi_j/\|\vphi_j\|\rangle|$ is constant for all $i\neq j$. We write $\conf{N}{n}$ (resp. $\tf{N}{n}$, $\etf{N}{n}$) for the space of all $N$-point configurations (resp. tight frames, equiangular tight frames) in $\CC^n$. 

The \emph{Gram matrix} (also called the \emph{Gramian}) of a frame  $\Phi$ is  the $N\times N$ matrix, $\Phi^*\Phi$, which stores the pairwise inner products, that is, its $(j, k)$ entry is $\vphi_j^*\vphi_k=\langle \vphi_j, \vphi_k\rangle$. The \emph{frame operator} (or frame matrix) is the $n\times n$ matrix  $\Phi\Phi^*$. One can check that $\Phi$ is a tight frame if and only if its frame operator $\Phi\Phi^*=A I_n$, where $I_n$ is the $n\times n$ identity matrix. This is also equivalent to $\tfrac{1}{A}\Phi^*\Phi$ being an idempotent of rank $n$. If $\Phi=\{\varphi_j\}_{j=0}^{N-1}$ is a tight frame and all its vectors $\varphi_j$ are unit vectors, then $A=\tfrac{N}{n}$.  

Although the norms of frame vectors can play a large role, such as frame scalings to achieve tightness~\cite{CKOPR15}, the concept of ETF is inherently projective, and many of the frame optimization problems can be thought of as the arrangement of lines in $\C^n$~\cite{CHS21, CGGKO}. Specifically, it is easy to see that the vectors in an ETF must have equal norm\footnote{Assume $|\langle\varphi_i, \varphi_j\rangle|=\alpha\|\varphi_i\|\|\varphi_j\|$ and let $S=\sum_j\|\varphi_j\|^2$. Combined with tightness,  we have $A\|\varphi_i\|^2=\|\varphi_i\|^4+\sum_{j\neq i}\alpha^2\|\varphi_i\|^2\|\varphi_j\|^2$ which simplifies to $(1-\alpha^2)\|\varphi_i\|^2=A-\alpha^2S$}. As such, and without any loss of generality, \textbf{we shall assume  each vector in any $N$-point configuration we consider is  unit-norm.}  

The coherence of an $N$-point configuration  $\Phi=\{\varphi_k\}_{k=0}^{N-1}$ of unit-vectors in $\C^n$, is defined to be 
\be
\mu(\Phi)=\max_{j\neq k\in \{0,1, \hdots, N-1\}}|\langle \varphi_j, \varphi_k\rangle|.
\ee
The Welch bound stipulates that 
\be\label{equ:welch}
\mu(\Phi)\geq \sqrt{\tfrac{N-n}{n(N-1)}}
\ee
 and this bound is achieved for $N$-configurations $\Phi=\{\varphi_k\}_{k=0}^{N-1}$ for which $|\langle \varphi_j, \varphi_k\rangle|= \sqrt{\tfrac{N-n}{n(N-1)}}$ for $j\neq k$, that is when $\Phi$ is equiangular. Furthermore, this bound can only be achieved when $n+1\leq N\leq n^2$ (assuming $n>1$). The search for ETFs is an active research area \cite{MR4162318, MR3549477, MR4024972, MR1984549}. Its applications include quantum information science and engineering \cite{MR4796067, MR2059685, MR4705312, MR4438049} and deep learning \cite{MR4250189, MR4412185, MR4455183}. In addition, the case $N=n^2$ is related to Zauner's conjecture \cite{appl2005, zaun1999, MR3983953}. 

 \subsection{Our contributions}\label{subsec1.2}
While the search for maximal ETFs that would solve Zauner's conjecture continues to be an active research area, the characterization of other ETFs whose symmetries are better understood has attracted a lot of interest in recent years. (There have been a lot of recent claims on the solution to Zauner's conjecture, see \cite{ApplebyFlammiaKopp2025ZaunerStark, ApplebyFlammiaKopp2026GhostSICs, FallonIverson2026ConferenceSignals, StarkReciprocity2026}).
Group frames, whose frame vectors are generated from a single vector by a linear action of a finite group, capture these symmetries well \cite{VW04, H07, VW16, IJM20, IJM20non}. Prototypical examples of such frames include Gabor and wavelet systems.
Recently, Fallon and Iverson~\cite{FI23} and Iverson et al.~~\cite{IJM24} proposed a method for generating an infinite family of $\etf{2n}{n}$. In particular, they look at the sub class of 2-circulant $\etfn$, which are generated by a unitary group action of the cyclic group $\ZZ/n$ on two fiducial vectors. They show that for $n\le 165$ there is a positive dimensional manifold of inequivalent $\etfn$s, and conjecture that this is true for every $n$. While 2-circulant $\etfn$s appear to be abundant, and approximate ones are easily found numerically by local methods \cite[Section 6]{IJM24}, their existence is established via a Newton--Kantorovich argument near numerical solutions, and explicit exact constructions remain scarce. The main motivation of this paper is to look for 2-circulant ETFs with more symmetry. The two examples we have in mind are cyclic or dihedral frames. While we show that cyclic $\etfn$ do not exist, dihedral ones do exist, for infinitely many values of $n$, and for most small ones. 

This paper focuses on group frames where the underlying unitary group representation can be \textbf{projective} (equation \eqref{equ:pi}). This notion is not new, and some early work can be found in \cite{HL08}.

Section \ref{sec2} lists the necessary algebraic notations and preliminaries, including Lemma \ref{lem:red_aut}, a group frame identification result involving the automorphism group of the Gram matrix. This lemma is used frequently in subsequent sections for frame characterization.\\

We first state a result when the group is abelian. Theorem \ref{thm:abelian} indicates that when looking for abelian group frames that are ETFs, we should focus our attention on projective representations. The rest of the paper then explores $\tf{2n}{n}$ and $\etfn$ who are dihedral group frames, as the dihedral group is the simplest non-abelian group. The general construction of dihedral configurations can be found in Section~\ref{sec:dihedral}. With Lemma \ref{lem:proj}, we give very explicit representations of dihedral group frames in \eqref{equ:SDnv} (strict) and \eqref{equ:PDnv} (projective).
Theorem~\ref{thm:circpart} characterizes the Gram matrix of any dihedral $\conf{2n}{n}$.
Theorem \ref{thm:str_g} and Theorem \ref{thm:proj_g} further characterize the Gram matrix of any dihedral $\tf{2n}{n}$.\\ 

Assuming regularity, which simply means that the first $n$ vectors in \eqref{equ:SDnv} or \eqref{equ:PDnv} span $\CC^n$, Theorem \ref{thm:main} gives a complete characterization of dihedral $\etfn$, which transfers the task of searching for dihedral $\etfn$ to searching for skew Hadamard matrices with a special structure. As it turns out, strict non-projective regular dihedral $\etfn$s do not exist, so Theorem \ref{thm:main} further breaks down to Theorem \ref{thm:2negcirc}. We also show that regular dihedral $\etfn$s do not exist for odd $n$. As for non-regular dihedral $\etfn$s, we do not know of even a single example, but we are able to prove in Theorem \ref{thm:finite} that in contrast to mere 2-circulant ETFs, there are only finitely many dihedral $\etfn$ up to switching equivalence.\\

As known examples of $\etfn$, we show in Theorem \ref{thm:paley} and Theorem \ref{thm:paley_d} that Paley $\etfn$ and its doubling are both genuine projective dihedral configurations. Finally, using characterization Theorem \ref{thm:main}, we are able to conduct an exhaustive search for small $n$ on a computer and therefore show that there is no regular dihedral $\etf{36}{18}$, and all the regular dihedral $\etfn$ for $n\leq20, n\neq16$ are of Paley type. There are regular dihedral $\etfn$ for $n=16, 24, 26$ that are not switching equivalent to a Paley type. We conclude with a full classification of all regular dihedral $\etfn$ for $n\leq22$ up to switching equivalence, given in Table \ref{tab:exact}.



\section{Notations, Motivation,  and Preliminaries}\label{sec2}
In this section, after setting up the notations used in the rest of the paper, we motivate our work by reviewing the construction of certain ETFs from appropriately chosen Hadamard matrices. We then introduce certain representations of the dihedral group that are necessary to state and prove our main results.

\subsection{Notations}\label{subsec2.1}
In this paper, $\omega_n:=\exp(-2\pi \i/n)$ will denote the standard primitive complex $n$th root of unity, and $\mu_n$ is the set of all $n$th roots of unity.
Let $S^1\subset \CC$ be the group of \emph{unimodular} numbers, i.e. those with modulus $1$. If $a$ and $b$ are two vectors or matrices, we write $a\doteq b$ if $a=\alpha b$ for a unimodular scalar $\alpha$, and say that $a$ is \emph{projectively equal} to $b$. The complex projective space $\CP^{d-1}$ is the set of all unit vectors $v\in \CC^{d}$ modulo the relation $\doteq$.

 Let $I_d$ denote the identity matrix of size $d\times d$. Sometimes we will  use $I$ if $d$ is clear. 
 For a complex vector $v$, let $\diag(v)$ denote the diagonal matrix with the diagonal $v$. 
 If $A_1,\ldots,A_m$ are scalars or matrices, we will write $\diag(A_1,\ldots,A_m)$ for the block diagonal matrix where $A_i$ are the blocks in the diagonal in the specified order.
 We shall write $[A,B;C,D]$ for the $2\times 2$ (block) matrix whose top row is $[A,B]$ and bottom row $[C,D]$.

Let $S_n$ denote the group of all permutations of the set $\underline{n}:=\{0,1,\ldots,n-1\}$, and $U(n)$ be the group of $n\times n$ unitary matrices. For any permutation $\sigma\in S_n$ we define the matrix $P_\sigma\in \CC^{n\times n}$ by the rule $(P_\sigma)_{i,j}=1$ if $i=\sigma(j)$ and $(P_\sigma)_{i,j}=0$ otherwise. The assignment $\sigma\mapsto P_\sigma$ is a group homomorphism.

A circulant $n\times n$ matrix is a matrix $C$ whose entries $C_{i,j}$ depend only on $(i-j)\mod n$. Hence, a circulant matrix is determined by its first row, which can be arbitrary. 
We let $\cir(v)$ denote the circulant matrix whose first row is $v$. 
The collection of all (complex, real, rational, integer) $n\times n$ circulant matrices is a commutative matrix algebra, also closed under the conjugate-transpose. In fact, any circulant matrix is a polynomial in the matrix $\cir([0,\ldots,0,1])$. A \emph{negacirculant} matrix (sometimes called \emph{negacyclic}) is a square matrix $N$ of the form $$N_{i,j}=\begin{cases} v_{j-i} & j\ge i\\ -v_{n+j-i} & j<i \end{cases},$$ for some vector $v=[v_0,\ldots,v_{n-1}]$. We denote this matrix by $\negc(v)$ as $v$ is its first row. Similarly to the circulant case, the collection of all (complex, real, rational, integer) negacirculant matrices is a commutative matrix algebra, also closed under the conjugate-transpose.

\begin{dfn}\label{def:2circ}
Let $H=[A, B; C, D]$ where all blocks $A, B, C, D$ are circulant (resp. negacirculant). We will call such $H$ a \emph{\bc\ (resp. \bnc)} matrix.
\end{dfn}

\subsection{A motivating result}\label{subsec2.2}
In this section we discuss the construction of some $\etfn$ from certain Hadamard matrices.  

A \emph{Hadamard matrix of order $m$} is a $\{\pm 1\}^{m\times m}$ matrix $H$ such that $HH^\top=mI$. A \emph{skew Hadamard matrix} is a Hadamard matrix $H$ such that $H+H^\top=2I_m$. A \emph{conference matrix} $C$ of order $m$ has 0 on the diagonal and $\pm1$ on the off-diagonal such that $CC^\top = (m-1)I_m$. It is easy to check that $H$ being skew Hadamard is equivalent to $H-I_m$ being a skew-symmetric conference matrix.
Hadamard matrices of orders at least 4 can exist only for size divisible by 4. It is conjectured that Hadamard matrices exist for any order divisible by 4.

It is well known that (skew-)symmetric conference matrices give rise to ETFs~\cite{DGS91}. We provide a proof for completeness.
\begin{lemma}\label{lem:shm}
   Suppose $n$ is an even natural number.  If $H$ is skew Hadamard of size $2n$, then $I_{2n}+\frac{\i}{\sqrt{2n-1}}(H-I)$ is the Gram matrix of an $\etfn$. Conversely, if $I_{2n}+\frac{\i}{\sqrt{2n-1}}C$ is the Gram matrix of an $\etfn$ where $C$ is real, then $C+I$ is skew Hadamard.
\end{lemma}
\begin{proof}
Let $C=H-I$ and $G=I+\i C/\sqrt{2n-1}$. Then $G$ is Hermitian if and only if $C$ is skew-symmetric. Moreover,
$$(\tfrac{1}{2}G)^2=\tfrac{1}{2}G\Longleftrightarrow (G-I)^2=I\Longleftrightarrow C^2=-(2n-1)I.$$
If $H$ is skew Hadamard, then the eigenvalues of $C$ reside in $\{\pm\i\sqrt{2n-1}\}$. Since $C$ has trace 0, both values appear with multiplicity $n$, resulting $G=I+\i C/\sqrt{2n-1}$ being of rank $n$, giving rise to an $\etfn$. The converse can be argued similarly using the facts listed above.
\end{proof}

Recently,  Fallon and Iverson~\cite{FI23} provide a new construction of $\etfn$ using a variant of Hadamard doubling, originally appearing in Wallis \cite{Wallis_1971}. A special case of this doubling technique is the following. Suppose that $H=I_{2n}+C$ is skew Hadamard. Then
\begin{equation}\label{equ:double}
H_2\ := \ \begin{bmatrix}
    I_{2n}+C & I_{2n}+C\\
    -I_{2n}+C & I_{2n}-C
\end{bmatrix}
\end{equation}
is again skew Hadamard. Hence the doubling \eqref{equ:double} yields an $\etf{4n}{2n}$.  

The framework introduced in \cite{FI23} transforms an $\etf{n}{\tfrac{n-1}{2}}$ into an $\etf{2n}{n}$. In a subsequent work~\cite{IJM24}, Iverson et al. constructed a new infinite family of $\etfn$ with a block circulant structure, that is the frame $\Phi$ is 2-circulant of the form $\Phi=[A, B]\in\etfn$ where both $A$ and $B$ are circulant matrices. Equivalently, $\Phi$ of the form 
\begin{equation}\label{equ:2c}
\Phi =[x, \Tc x, \cdots, \Tc^{n-1}x, y, \Tc y, \cdots, \Tc^{n-1}y],
\end{equation}
where $\Tc$ denotes the circular translation operator and $x, y $ are unit-norm vectors, which are termed \emph{fiducial vectors}. The notion of a 2-circulant configuration is the same as a \emph{2-generator $\ZZ/n$-harmonic configuration}, as per ~\cite[Definition 21]{IJM24}, where $\ZZ/n$ is the cyclic group of order $n$. The doubling construction above does not preserve the property of being 2-circulant. Next, we introduce the notion of dihedral configurations, which in particular, after suitable transformations, are 2-circulant.


\subsection{Group and dihedral configurations}\label{sec:dihedral}

Let $G$ be a finite abstract group. A \emph{unitary representation} of $G$ of dimension $n$ is a homomorphism $\rho:G\to U(n)$, that is a map satisfying 
\begin{equation}\label{equ:rho}
\rho(gg')=\rho(g)\rho(g') \text{ for all }g,g'\in G.
\end{equation}  
A \emph{projective unitary representation} of $G$ of dimension $n$ is a map $\pi:G\to U(n)$, satisfying 
\begin{equation}\label{equ:pi}
\pi(gg')\doteq \pi(g)\pi(g') \text{ for all } g,g'\in G.
\end{equation}
Projective representations generalize ordinary representations, and we will refer to a representation $\rho$ satisfying \eqref{equ:rho} as a \emph{strict} representation. If $\rho$ is not strict, then we will refer to it as a \emph{genuine projective} representation.

\begin{dfn}\label{def:gframe}
Given a finite group $G$ with an ordering, and a representation $\pi$ (strict or projective), and a unit vector $v\in \CC^n$, which we consider as the \emph{fiducial} vector, we define the configuration $\pi(G)v$ to be the ordered multiset $\{\pi(g)v\ | \ g\in G\}$. We will usually denote this set merely by $Gv$, suppressing the representation $\pi$, whenever it is understood from the context. We shall call such sets \emph{$G$-configurations}. If $Gv$ is a frame we will say that it is a \emph{$G$-frame}. When talking in general about such frames without reference to the group $G$, we will use the term \emph{group configurations} and \emph{group frames}.
\end{dfn}
If the ordering does not matter, any other vector in $Gv$ may be taken as the fiducial vector.

\noindent We are mostly interested in  \emph{dihedral configurations}  for which the group $G$ is the dihedral group. 
Recall that the dihedral group is the abstract group $D_n$ defined by two generators $\mu,\tau$ and the relations $\mu^n=\tau^2=1$ and $\tau\mu=\mu^{-1}\tau$. This is a group of order $2n$ with the set of elements $\{1,\mu,\ldots,\mu^{n-1},\tau,\mu\tau,\ldots,\mu^{n-1}\tau\}$. 
A well-known strict representation in dimension $2$ is defined by 
$$\rho(\mu) = \begin{bmatrix} \cos(2\pi/n)& \sin(2\pi/n)\\ -\sin(2\pi/n)&\cos(2\pi/n)\end{bmatrix}$$ and $\rho(\tau)=\diag(-1,1)$, viewed as the symmetry group of the perfect $n$-gon. 
Since we are interested in its action on $v\in\C^n$,
 a strict unitary representation of $D_n$ is obtained by  giving two unitary matrices $M,T\in\C^{n\times n}$ satisfying the same relations; $M^n=T^2=I$ and $TM=M^{-1}T$. Indeed, we can define the representation by $\rho(\mu^i\tau^j):=M^iT^j$ for $0\le i\le n-1$ and $j=0,1$. For example, we can take
 \begin{equation}\label{equ:MT}
 M=\diag(1,\omega_n, \omega_n^2,\cdots, \omega_n^{n-1}),\quad T =\begin{bmatrix}
1 &  &  &  &  \\
  &  &  &  & 1\\
  &  &  & 1 &  \\
  &  & \iddots &  &  \\
  & 1 &  &  &  
\end{bmatrix}.\end{equation}
With this strict representation, we are interested in frames of the form
 \begin{equation}\label{equ:Dnv}
 [v, Mv, \cdots, M^{n-1}v, Tv, MTv, \cdots, M^{n-1}Tv],
 \end{equation}
 for some unit-norm vector $v$. If we rotate this frame by the unitary matrix of the discrete Fourier transform $\Fc$, we obtain a special example of a 2-circulant frame, as given in \eqref{equ:2c}:
 \begin{equation}\label{equ:dft}
 \begin{split}
 &\Fc[x, \Tc x, \cdots, \Tc^{n-1}x,\ y, \Tc y, \cdots, \Tc^{n-1}y] = [v, Mv, \cdots, M^{n-1}v,\ Tv, MTv, \cdots, M^{n-1}Tv],\\
 &\text{where }x = \Fc^*v \text{ and } y=\Fc^*Tv=Tx.
 \end{split}
 \end{equation}
 
 Likewise to give a projective representation amounts to giving two matrices $M,T$ satisfying $M^n\doteq T^2\doteq I$ and $TM\doteq M^{-1}T$. 
Up to equivalence of projective representations we can be yet more specific (see Lemma \ref{lem:proj}). Some of the projective representations of the dihedral group also give rise to 2-circulant frames. We will see that, up to switching equivalence (See, Definition~\ref{deft:switch}), regular dihedral frames are such 2-circulant frames. 

 The doubling construction \eqref{equ:double} does not preserve the 2-circulant or dihedral structures. However, we will encounter below cases where dihedrality is preserved by doubling.


\subsection{Switching equivalences and automorphism groups}\label{sec:switching}

A phase \emph{monomial matrix} is a square matrix $\Pi\in \CC^{N\times N}$ which can be represented (uniquely) as a product $\Pi=DP$ where $D,P\in \CC^{N\times N}$ such that $D$ is a diagonal matrix with $|D_{i,i}|=1$ for all $i$, and $P$ is a permutation matrix. The collection of all $N\times N$ phase monomial matrices will be denoted by $\mon(N)$, and forms a group under matrix multiplication. The group $\mon(N)$ is the semidirect product of the subgroup of permutation matrices $S_N$ by the normal subgroup of the diagonal phase matrices.

\begin{dfn}[Switching automorphism groups]\label{deft:switch}
    Let $\Phi, \Psi\in \CC^{n\times N}$ be any two configurations. We say that $\Phi,\Psi$ are \emph{switching equivalent} if there exist a phase monomial matrix $\Pi\in \mon(N)$ and a unitary matrix $Q\in U(n)$ such that $$\Psi=Q\Phi \Pi^*.$$ We denote this by $\Phi\sim\Psi$.  
\end{dfn}
Switching equivalence is an equivalence relation. In general the matrix $\Pi$ does not determine $Q$, but when $\Phi$ is a frame it does. If the pair $(\Pi,Q)$ exhibits an equivalence between $\Phi$ and $\Psi$, then also $(\alpha\Pi,\alpha Q)$ for $|\alpha|=1$ does; informally we would like to think of them as being the \emph{same} equivalence. This motivates the following definition.
\begin{dfn}
    Let $\Phi\in \CC^{n\times N}$ be a configuration. The \emph{switching automorphism group} $\Aut(\Phi)$ is the sub quotient group of all pairs $(\Pi,Q)\in \mon(N)\times U(n)$ satisfying \be\label{eq:aut_def} Q\Phi \Pi^*=\Phi,\ee modulo the identification $(\Pi,Q)\approx(\alpha \Pi,\alpha Q)$ for $|\alpha|=1$. 
\end{dfn}
When $\Phi\in \CC^{n\times N}$ is a frame, it is useful to work with the Gram matrix $\Phi^*\Phi\in \CC^{N\times N}$. Given a switching automorphism $(\Pi,Q)$ of $\Phi$, the Gram is invariant under the conjugation by $\Pi$, and $Q$ can be uniquely recovered from $\Pi$. Converesely, if the Gram is invariant under some $\Pi\in \mon(N)$, then there exists a $Q\in U(n)$ such that $(\Pi,Q)$ is a switching equivalence. This motivates the following definition.

\begin{dfn}[Switching automorphism groups for Gram matrices]
For a positive semidefinite matrix $S\in \CC^{N\times N}$, the \emph{switching automorphism group} $\Aut(S)$ is the sub quotient group of all $\Pi\in \mon(N)$ such that $\Pi S\Pi^*=S$, modulo the identification $\Pi\approx\alpha \Pi$ for $|\alpha|=1$. 
\end{dfn}
For frames $\Phi$ there is a canonical identification $\Aut(\Phi)=\Aut(\Phi^*\Phi)$ given by $(\Pi,Q)\longleftrightarrow \Pi$.

\begin{ex}

The configuration $\Phi=\{a,b,c,d\}$ in the picture has $\Aut(\Phi)$  containing a cyclic of order $4$.  We apply the rotation $U$ counterclockwise by angle $\pi/4$, then multiply $e=Ud$ by $-1$, and finally apply a cyclic permutation to go back to $\Phi$. We can also reflect about $c$ and get another automorphism. We have $\Aut(\Phi)\cong D_8$, the dihedral group.
	
	
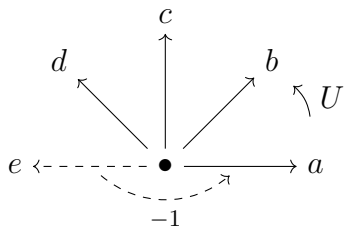
\begin{figure}[h]\label{fig:1}
		\begin{center}
			\begin{tikzpicture}[node distance=4cm]
			\node (O) at (0,0) {\textbullet};
			\node (B) at (2*0.707,2*0.707) {$b$ };
			\node (C) at (0,2) {$c$ };
			\node (D) at (2*-0.707,2*0.707) {$d$ };
			\node (E) at (-2,0) {$e$ };
			\node (A) at (2,0){$a$ };
			
			\node (P) at (2*0.965,2*0.259) {};
			\node (Q) at (2*0.866,2*0.5) { };
			
			\node (R) at (-1,0) { };
			\node (S) at (1,0) { };
			
			\draw[->] (O) edge (A) (O) edge (B) (O) edge (C) (O) edge (D);
			\draw[dashed,->] (O) edge (E);
			
			\draw[bend right=20,->] (P) edge  ($ (P)!0.6cm!(Q) $) node at (2.4*0.924,2.4*0.382) {$U$};
			\draw[bend right=40,,dashed,->]  (R) edge (S) node at (0,-0.7) {\scriptsize $-1$};
			\end{tikzpicture}
		\end{center}
	\end{figure}
\end{ex}

\subsection{More on group configurations}

A special case of interest is when the vectors in $Gv$ are pairwise non-proportional. 
Sometimes we will want to remove from $Gv$ linearly dependent copies of vectors and leave just one representative per each dependency class. Let $\overline{Gv}$ denote any maximal subset of $Gv$ containing only non-proportional vectors. We call $\overline{Gv}$ a reduced $G$-configuration. We call $Gv$ \emph{reduced} if as sets $Gv = \overline{Gv}$, i.e. there are no repetitions in $Gv$ up to proportionality.

\begin{ex}\label{ex:Wh}
    The group $G=(\ZZ/n\ZZ)^2$ has a 2-dimensional unitary projective representation $\pi:G\to U(n)$ defined by $\pi(1,0)=\diag(1,\omega,\omega^2,\ldots,\omega^{n-1})$, $\pi(0,1)=\cir(0,0,\ldots,1)$ and $\pi(i,j):=\pi(1,0)^i\pi(0,1)^j$. For a nonzero vector $v\in \CC^n$, the $G$-cyclic frame $Gv$ is called a \emph{Gabor} frame. Zauner's conjecture states that for every $n$ there exists a fiducial $v$ such that $Gv$ is an $\etf{n^2}{n}$~\cite{zaun1999}. There are also maximal MUBs that can be constructed from Gabor frames \cite{MM19}.
\end{ex}

There is a group homomorphism $G\to \Aut(\overline{Gv})$, given by the following recipe. Write $\overline{Gv}=\{\pi(g_1)v,\ldots,\pi(g_n)v\}$. For every $g\in G$ there is a unique permutation $\tau\in S_n$ and unique scalars $\alpha_i\in \CC^\times$, $1\le i\le n$ such that $\pi(g)\pi(g_i)v=\alpha_i \pi(g_{\tau(i)})v$. Then we map $g$ to the pair $(\pi(g),\diag(\alpha)P_{\tau})$, which gives rise to an automorphism of $\overline{Gv}$. This mapping is clearly a group homomorphism (Recall that we define automorphisms as pairs modulo unimodular scalars).
This homomorphism in general need not be injective. The following is a useful criterion to identify group frames.

\begin{lemma}\label{lem:red_aut}
    Suppose that $\Phi=\{\varphi_j\}_{j=0}^{N-1}\in \conf{N}{n}$ is reduced and $G$ is a group of order $N$ with an embedding $\iota:G\hookrightarrow \Aut(\Phi^*\Phi)$. If $\iota(G)\subseteq Mon(N)/\sim$ acts transitively on the set $\underline{N}$, then $\Phi$ is switching equivalent to a reduced group frame $\overline{Gv}$.
\end{lemma}

\begin{proof}
    Let us lift the homomorphism $\iota$ to a map $\hat\iota:G\to Mon(N)$, which is clearly a projective representation of $G$. For every $g\in G$, the frame matrix $\Phi^g:=\Phi \cdot \hat\iota(g)^*$ admits the same Gram matrix, since $\hat\iota(g)\Phi^*\Phi\hat\iota(g)^*=\Phi^*\Phi$ by the definition of an automorphism. As a consequence, there exists a matrix $\pi(g)\in U(n)$ such that \be\label{eq:pig} \forall g\in G, \ \ \pi(g)\Phi^g=\Phi.\ee Since $\Phi$ is a frame, $\pi(g)$ is unique up to a scalar, and let us fix a choice of the function $\pi(g)$. It follows that $\pi:G\to U(n)$ is a unitary projective representation. Fix some ordering on $G=\{1=g_0\le \cdots \le g_{n-1}\}$ and consider the vector $v=\varphi_0$ and the configuration $Gv=\{v,g_1v,\ldots,g_{N-1}v\}$ via $\pi$. From \eqref{eq:pig} it follows that for each $i$, $g_iv\doteq \varphi_j$ for some $j=j(i)$. By the transitivity of $\iota(G)$, the fact that $|G|=N$, and that $\Phi$ is reduced, it follows that the assignment $i\mapsto j(i)$ is an isomorphism of $\underline{N}$. In particular, $\Phi$ and $\overline{Gv}$ differ by ordering and scalings, hence are switching equivalent.
\end{proof}

\subsection{Strict and Projective unitary representations}
In this section, we further elaborate on the notions of strict and projective representations given in \eqref{equ:rho} and \eqref{equ:pi}.

\begin{dfn}
    Two (strict) unitary representations $\rho,\rho':G\to U(n)$ are said to be \emph{equivalent}, if there exists a unitary $Q\in U(n)$ such that $Q\rho(g) Q^*=\rho'(g)$ for all $g\in G$. Two projective unitary representations $\pi,\pi':G\to U(n)$ are \emph{equivalent} if there exists a unitary $Q\in U(n)$ such that $Q\pi(g) Q^*\doteq\pi'(g)$ for all $g\in G$. We call $Q$ an \emph{equivalence matrix}. 
\end{dfn}

\begin{dfn}
    A map (or homomorphism) between two strict representations $\rho:G\to U(n),\rho':G\to U(m)$ (resp. projective representations $\pi,\pi'$) is a linear map $T:\CC^n\to \CC^m$ such that for all $g\in G$, $T\circ \rho(g)=\rho'(g)\circ T$ (resp. $T\circ \pi(g)\doteq \pi'(g)\circ T$).
    We shall write this as $T:\rho\to \rho'$ (resp. $T:\pi\to \pi$). An \emph{isomorphism} of (projective) unitary representations is a bijective homomorphism. Clearly an equivalence is an isomorphism. Not every equivalence is an isomorphism, but isomorphic representations are equivalent.
\end{dfn}

\begin{ex}
    Projective representations naturally arise from automorphism groups of configurations. If $\Phi$ is an $n\times N$ frame, every $g\in \Aut(\Phi)$ corresponds to a pair $(Q(g),\Pi(g))\in U(n)\times Mon(N)$ satisfying $Q(g)\Phi\Pi(g)^*=\Phi$. This correspondence is not a function, as pairs are being identified modulo scalars. But \emph{choosing} a function $g\mapsto (Q(g),\Pi(g))$, the maps $g\mapsto Q(g)$ and $g\mapsto \Pi(g)$ define projective representations of $\Aut(\Phi)$, whose equivalence classes are independent of the choice. Moreover, if $\Phi$ is a frame, then to each $\Pi(g)$ there corresponds a unique $Q(g)$ and this correspondence satisfies $Q(g)\Phi=\Phi\Pi(g)$. Then by definition,  $\Phi$ is a homomorphism from $\Pi$ to $Q$. Similarly, $\Phi^*:Q\to \Pi$. If $\Phi$ is a tight frame, then $\Phi$ is a left inverse of $\Phi^*$, up to a scalar, which shows that $\Pi$ contains an isomorphic copy of $Q$ as a subrepresentation, and $\Phi^*\Phi$ is up to a scalar the projection onto this subrepresentation. Everything said restricts to subgroups of $\Aut(\Phi)$ or more generally to groups with homomorphisms into $\Aut(\Phi)$.
\end{ex}

We note here the connection between projective representations of a finite group $G$ and its second cohomology group. For more details, see \cite{Karpilovsky1985,Karpilovsky1987}.
\begin{dfn}
    A \emph{2-cocycle} is a function $f:G\times G\to \CC^\times$ satisfying
    $$ f(y,z)f(x,yz)=f(xy,z)f(x,y), \ \ \text{for all } x,y,z\in G.$$
    A \emph{2-coboundary} is a function $c:G\times G\to \CC^\times$ for which there exists a function $h:G\to \CC^\times$ such that for all $x,y\in G$,
    $$c(x,y)=h(x)h(y)h(xy)^{-1}.$$ 
    The set $Z^2(G;\CC^\times)$ of 2-cocycles is a group under pointwise function multiplication. The set of 2-coboundaries $B^2(G;\CC^\times)$ is a subgroup of $Z^2(G;\CC^\times)$. The \emph{2nd cohomology} group is the quotient $$H^2(G;\CC^\times)\ := \ \frac{Z^2(G;\CC^\times)}{B^2(G;\CC^\times)}.$$ Two 2-cocycles $f,f'$ are said to be \emph{cohomologous} if $f=f'\cdot c$ for a 2-coboundary $c$. The image of $f$ in $H^2(G;\CC^\times)$ is called the \emph{cohomology class} of $f$, denoted by $[f]$.
\end{dfn}

If $\pi:G\to U(n)$ is a unitary projective representation, then we have $\pi(xy)=f(x,y)\pi(x)\pi(y)$ for some phase $f:G\times G\to \CC^\times$. The function $f=f_\pi$ is a 2-cocycle. If $\pi$ and $\pi'$ are equivalent projective representations, then $f_\pi$ and $f_{\pi'}$ are cohomologous. Therefore $[f_\pi]$ is an invariant of equivalence classes of representations. A representation $\pi$ is equivalent to a strict representation if and only if $[f_\pi]$ is the trivial cohomology class.
\begin{dfn}
    The cohomology class $[f_\pi]$ attached to a projective representation $\pi$ is called the \emph{Schur multiplier} of $\pi$.
\end{dfn}

\begin{ex}
    The representation of Example \ref{ex:Wh} corresponds to the 2-cocycle $f:(\ZZ/n)^2\times (\ZZ/n)^2\to \C^\times$ given by $f(a,b;c,d)=\omega^{bc}$. This can be shown to be cohomologous to the symplectic cocycle $f'(a,b;c,d)=\exp(\pi \i/n \cdot \det(a,b;c,d))$.
\end{ex}


\section{Abelian Configurations}\label{sec3}

A configuration $\Phi$ is \emph{strictly abelian} with respect to $(G,\rho)$, if it equals $\overline{\rho(G)v}$ where $\rho$ is a strict unitary representation of a finite abelian group $G$. In this section, we show that there are no strictly abelian $\etf{N}{n}$ for most combinations $n$ and $N$.  Given a strictly abelian reduced configuration $\Phi=\overline{\rho(G)v}$ generated by a finite abelian group $G$ together with a strict representation $\rho$, there is a filtration of subgroups $K\subseteq H\subseteq G$ such that $K=\ker\rho$ is the kernel, i.e. the set of all elements $g\in G$ for which $\rho(g)=I$, and $H=\{g\in G\ | \ \exists\lambda\in \CC,\ \rho(g)v=\lambda v\}$, where $v$ is any given vector in $\Phi$. The definition of $H$ does not depend on the choice of $v\in \Phi$, since $G$ is abelian. The next lemma shows that the group $G$ can be reduced to the size of $\Phi$.

\begin{lemma}
    Let $\Phi\in \conf{N}{n}$ be a strictly abelian \emph{reduced frame} with respect to $(G,\rho)$, and let $K\subseteq H\subseteq G$ be the aforementioned filtration. Then we have $|G/H|=|\Phi|$, and there exists a strict representation $\rho':G/H\to U(n)$ generating $\Phi$.
\end{lemma}

\begin{proof} The fact that $|G/H|=|\Phi|$ follows from the fact that $G$ acts transitively on the image of $\Phi$ in $\CC\PP^{n-1}$ and $H$ stabilizes the image of $v$.
    Choose a unit vector $v\in \Phi$. Then there exists a function $\lambda:H\to \CC^\times$, independent of the choice of $v$,  such that for all $h\in H$, $\rho(h)v=\lambda(h)v$. Clearly $\lambda$ is a homomorphism. It is well-known that since $\CC^\times$ is a divisible group, it is an injective $\ZZ$-module, and hence $\lambda$ extends to a homomorphism $\chi:G\to \CC^\times$, see \cite[Corollary 2.3.2]{Weibel1994}. We define a new (strict) representation $\rho_0:G\to U(n)$, given by $\rho_0(g)=\chi(g)^{-1}\rho(g)$. Then $\rho_0|_H$ is trivial on $v$, as well as any other vector in $\Phi$, and since $\Phi$ is a frame, this shows that $\rho_0|_H$ is the trivial representation. Hence $\rho_0$ descends to a representation $\rho':G/H\to U(n)$, and for all $g\in G$, $\rho(g)v\doteq \rho_0(g)v$, which implies that $\Phi=\overline{\rho'(G/H)v}$.
\end{proof}

We can now prove:
\begin{thm}\label{thm:abelian}
    There are no strictly abelian $\etf{N}{n}$ whenever $n(N-n)/(N-1)$ is not an integer.
\end{thm}

\begin{proof}
    Let $\Phi=\overline{{\rho}(G)v}$ be a configuration, for $G$ abelian and $\rho$ strict. By the previous lemma, we may reduce to the case where $|G|=|\Phi|$ and $\Phi=\rho(G)v$. 
  We index the Gram matrix $\Phi^*\Phi$ by the elements of the group $G$ and as a consequence the $(x,y)$-element is $\langle \rho(x)v,\rho(y)v\rangle=\langle v,\rho(x^{-1}y)v\rangle=f(x^{-1}y)$ for some function $f:G\to \CC$. The collection $$\mathcal A_G:=\{M\ | \exists f:G\to \CC ,\ M_{x,y}=f(x^{-1}y)\}$$ is a matrix algebra, i.e. closed under matrix addition and multiplication. Moreover, it is isomorphic to the group ring $\CC[G]$, with the isomorphism given by $\left(f(x^{-1}y)\right)_{x,y} \mapsto \sum_{g\in G} f(g^{-1})g$. 
    This map is clearly a bijection, and the compatibility with addition and multiplication can be readily checked. It follows that $\mathcal A_G$ is a commutative matrix algebra.

    The minimal idempotents of $\CC[G]$ are given by the (multiplicative) characters of $G$, namely for a character $\chi:G\to \CC^\times$, they are given by $e_\chi=\tfrac{1}{|G|}\sum_{g\in G} \chi(g^{-1})g$, and corresponding to them are that matrices $E_\chi:=\tfrac{1}{|G|}\left(\chi(x^{-1}y)\right)_{x,y}$. Every other idempotent is a sum of distinct minimal idempotents in a unique way. Suppose now that $\Phi$ is a TF of size $n\times N$, $N=|G|$. Then we conclude that 
    $$\Phi^*\Phi=\frac{N}{n}\sum_i E_{\chi_i}, \ \ \text{for some characters}\ \chi_i.$$
     In particular $n\Phi^*\Phi$ is a matrix whose coefficients are algebraic integers, since the values of the characters $\chi_i$ are roots of unity. The absolute value of a complex algebraic integer is again an algebraic integer. If $\Phi$ is an ETF, then the modulus of an off-diagonal coefficient of $n\Phi^*\Phi$ is $n$ times the Welch bound \eqref{equ:welch}, which is $\sqrt{n(N-n)/(N-1)}$. The number $n(N-n)/(N-1)$ is not an integer by assumption, thus not an algebraic integer, so we get a contradiction.
\end{proof}

Note that in the special case where $|G|=|\Phi|=N$, i.e.\ $\Phi$ is a harmonic frame, Theorem~\ref{thm:abelian}
is well known: a harmonic $\etf{N}{n}$ arises from an $(N,n,\lambda)$ difference set in $G$, see
\cite[Theorems 11.1 and 12.3]{Waldronbook}, where $\lambda=\frac{n(n-1)}{N-1}=n-\frac{n(N-n)}{N-1}$
must be an integer. Theorem~\ref{thm:abelian} extends this to configurations generated by an
arbitrary finite abelian group, and its proof uses a
different technique.

\begin{cor}
    There is no strict or projective cyclic group $\etf{N}{n}$ whenever $n(N-n)/(N-1)$ is not an integer.
\end{cor}
\begin{proof}
    Such a frame is of the form $\Phi=\{v,Xv,\ldots,X^{n-1}v\}$ for some unitary matrix $X$ such that  $X^n=\lambda I$, where $\lambda$ is a scalar. We can modify $X\mapsto \lambda^{-1/n}X$, and assume without loss of generality that $X^n=I$. It follows that $\Phi$ is a strict cyclic group frame, and we conclude the proof of the result by appealing to Theorem \ref{thm:abelian}.
\end{proof}

\begin{remark}
Theorem \ref{thm:abelian} rules out the existence of a strictly abelian $\etf{N}{n}$ for many $N, n$. In particular, there are no strict abelian $\etf{n^2}{n}$ nor $\etf{2n}{n}$.
\end{remark}
\begin{remark}
    Theorem \ref{thm:abelian} is no longer true when $\rho$ is genuinely projective. The famous example is the group $G=(\ZZ/n)^2$ with Zauner's $\etf{n^2}{n}$ which is known to exist for many $n$, and conjectured to exist for all $n$.
\end{remark}


\section{Characterization of dihedral $\etfn$}\label{sec4}

In this section, we initiate a characterization of  dihedral $\etfn$. We think of the dihedral group as the simplest example of a non abelian group, and seek positive results as opposed to the negative result we obtained in Section~\ref{sec3}. First, we characterize the structure of the Gram matrix of a dihedral configuration, strict as well non-strict. Next, we give a characterization of dihedral tight frames. Subsequently, we characterize regular dihedral $\etfn$, proving that they do not exist in the strict case, and that they all come from certain structured skew Hadamard matrices in the sense of \eqref{eq:shm2etf} in the projective case. Finally, we turn to the general (non-regular) case and show that, up to switching equivalence, there are only finitely many dihedral $\etfn$ for fixed $n$.

\begin{lemma}\label{lem:proj}
    If $n$ is odd, then every projective unitary representation of $D_n$ is equivalent to a strict representation. If $n$ is even, then
    every projective unitary representation not equivalent to a strict representation is equivalent to a representation $\pi$ such that  $\pi(\mu)=M$ and $\pi(\tau)=T$ with $M^n=-I$, $T^2=I$ and $TM=M^{-1}T$. Conversely, when $n$ is even every projective representation with $M^n=-I$, $T^2=I$ and $TM=M^{-1}T$ is not equivalent to a strict representation.
\end{lemma}

\begin{proof}
Suppose that we are given a projective unitary representation $\pi':G\to U(n)$, and let $M'=\pi'(\mu)$ and $T'=\pi'(\tau)$. Then $T^{'2}=\alpha I$ for some $\alpha\in S^1$. Let $T=\beta T'$ with $\beta^2=\alpha^{-1}$. Then $T^2=I$. Likewise $M^{'n}=\gamma I$ for some $\gamma\in S^1$ and putting $\delta^n=\gamma^{-1}$ we define $M=\delta M'$ and have $M^n=I$. This defines an equivalent projective representation $\pi$ with $\pi(\mu)=M$ and $\pi(\tau)=T$. We have $TMT^{-1}=\varepsilon M^{-1}$ for some $\varepsilon\in S^1$. Raising this identity to the power $n$ shows that $\varepsilon^n=1$. If we modify $M\mapsto \eta M$ with $\eta^n=1$ we obtain the same identities, but with $\varepsilon$ being replaced by $\varepsilon\eta^2$. Thus when $n$ is odd we can arrive at $\varepsilon\eta^2=1$, proving the first part of the lemma.

Suppose now that $n$ is even. Then, without loss of generality, we may reduce to the case where $\varepsilon=\pm 1$. If $\varepsilon=1$ then $\pi$ is a strict representation. Otherwise $\varepsilon=-1$. Modifying $M\mapsto \exp(\pi \i/n)M$ we obtain the relations $M^n=-I$ and $TM=M^{-1}T$ as desired, proving the second part of the lemma. Finally suppose that $\pi$ is a projective representation with $M,T$ satisfying these relations. Let $M'=\alpha UMU^*,T'=\beta UTU^*$ with unitary $U$ and $\alpha,\beta\in S^1$. Then $T^{'2}=\alpha^2I$, $M^{'2}=-\beta^n I$ and $T'M'T^{'-1}=\beta^{2}M^{'-1}$. For $M',T'$ to satisfy the strict dihedral relations, we must have both $\beta^n=-1$ and $\beta^2=1$, which is impossible when $n$ is even.
\end{proof}

\begin{remark}\label{rem:coh}
     A consequence of Lemma \ref{lem:proj} is that when $n$ is even the cohomology group $H^2(D_n;\CC^\times)$ has precisely 2 elements, and when $n$ is odd it is the trivial group, see also \cite{Karpilovsky1987}[Proposition 2.11.4].
\end{remark}

Lemma \ref{lem:proj} simplifies the dihedral configuration situation. If the representation is strict, then it is equivalent to choosing two matrices $M,T\in U(n)$ satisfying the relation $M^n=T^2=I;\ TM=M^{-1}T$. The choice in \eqref{equ:MT} is an example, but it is not the only one up to equivalence of representations.
Having done so, we can generate the configurations 
 \begin{equation}\label{equ:SDnv}
 SD_nv = [v, Mv, \cdots, M^{n-1}v, Tv, MTv, \cdots, M^{n-1}Tv]\in \conf{2n}{n},
 \end{equation}
 which we term \emph{strict dihedral} configurations.
 
 On the other hand, if the representation is genuine projective (that is, not equivalent to a strict one), it boils down to choosing $\hat M, \hat T\in U(n)$ such that $\hat M^n = -I_n, \hat T^2=I_n, \hat T\hat M=\hat M^{-1}\hat T$. From this we can generate configurations 
\begin{equation}\label{equ:PDnv}
PD_nv=[v, \hat Mv, \cdots, \hat M^{n-1}v, \hat Tv, \hat M\hat Tv, \cdots, \hat M^{n-1}\hat Tv] \in \conf{2n}{n},
\end{equation}
which we term \emph{genuine projective} configurations. As an example, we can choose
\begin{equation}\label{equ:MT2}
\hat M= \diag(\omega_{2n}, \omega_{2n}^3,\cdots, \omega_{2n}^{2n-1}),\quad \hat T =\begin{bmatrix}
   & & & & 1\\
   & & & 1 & \\
&\\
  1 &  & & & 
\end{bmatrix}.
\end{equation}

In the sequel, $D_nv$ will denote a dihedral configuration, which can be strict or genuine projective.
 If $\pi$ is a projective representation of $D_n$ equivalent to a strict representation, we will say that $D_nv$ is a \emph{strict} dihedral configuration and denote it by $SD_nv$ (\eqref{equ:SDnv} being a particular switching equivalent realization). If $\pi$ is not equivalent to a strict representation, we will say that $D_nv$ is a \emph{(genuine) projective} dihedral configuration and denote it by $PD_nv$ (\eqref{equ:PDnv} being a particular switching equivalent realization). 


\subsection{Gram matrices of dihedral configurations}

\begin{thm}\label{thm:circpart}
Let $\Phi=D_nv$ be a dihedral $\conf{2n}{n}$. Up to switching equivalence the Gram matrix $\Phi^*\Phi$ can be partitioned as 
\be \label{eq:circpart} 
\Phi^*\Phi \ = \ \begin{bmatrix}
        A & B\\ B^\top & A^\top
    \end{bmatrix}, \quad\text{$A$ hermitian, and $B$ real.}
 \ee
 Moreover, if $\Phi$ is strict (resp. projective), then both $A$ and $B$ can be taken to be circulant (resp. negacirculant). 
 Conversely, if $\Phi$ is reduced and $\Phi^*\Phi$ is switching equivalent to a matrix of the form \eqref{eq:circpart}, with $A,B$ being circulant (resp. negacirculant), then it is a strict (resp. projective) dihedral configuration. 
\end{thm}

\begin{proof}
We first treat the strict case. Let $\rho$ be a strict representation of $D_n$ generating $\Phi$. Write $M=\rho(\mu)$ and $T=\rho(\tau)$. Then $M^n=T^2=I$ and $MT=TM^{-1}$. We further reorder $\Phi=\{v,Mv,\ldots,M^{n-1}v,Tv,MTv,\ldots,M^{n-1}Tv\}$ if necessary. With respect to this ordering we partition the Gram $\Phi^*\Phi=[A,B;C,D]$ into four $n\times n$ blocks, each block is indexed by $0,1,\ldots,n-1$. Then $A_{i,j}=\langle M^iv,M^jv\rangle=\langle v,M^{j-i}v\rangle$, showing that $A$ is circulant. Similarly $B_{i,j}=\langle M^iv,M^jTv\rangle=\langle v,M^{j-i}Tv\rangle$, and $B$ is circulant as well. We then compute $C_{i,j}=\langle M^iTv,M^jv\rangle=\langle TM^{-i}v,M^jv\rangle=\langle M^{-i}v,TM^jv\rangle=\langle M^{-i}v,M^{-j}Tv\rangle=\langle v,M^{i-j}Tv\rangle=B_{j,i}$, and $C=B^\top$. Similarly we show that $D=A^\top$. Since the Gram is hermitian, this must imply that $B^\top=B^*$, which shows that $B$ is real. Likewise $A$ is hermitian.

 In the projective case we proceed similarly. By  Lemma \ref{lem:proj} we have $\Phi = PD_nv$ as shown in \eqref{equ:PDnv}.
We partition the Gram  matrix similarly and notice that $A_{i,j}=\langle v,\hat M^{j-i}v\rangle$. Using $\hat M^n=-I$ we see that $A_{i,j}=A_{0,j-i}$ if $i\le j$, and $A_{i,j}=-A_{0,n+j-i}$ if $i>j$, which is by definition a negacirculant matrix. The analysis for $B,C,D$ is similar, leaving the details to the reader.\\

Conversely, suppose that $\Phi^*\Phi=[A,B;B^\top,A^\top]$ with $A,B$ circulant, $A$ hermitian and $B$ real. Consider the strict representation $\rho:D_n\to U(2n)$, which is the permutation representation given by the left multiplication of $D_n$ on itself. We consider $\rho$ given with respect to the ordering $1,\mu,\ldots,\mu^{n-1},\tau,\mu\tau,\ldots,\mu^{n-1}\tau$. Then $\rho(\mu)=\cir([0,\ldots,0,1])\oplus \cir([0,\ldots,0,1])$ and $\rho(\tau)$ is the permutation matrix of the permutation $(0,n)(1,2n-1)(2,2n-2)\cdots(n-1,n+1)$. The matrices that centralize $\rho(\mu)$ are precisely the matrices $[A,B;C,D]$ with $A,B,C,D$ circulant of size $n\times n$. Among those, the matrices that in addition centralize $\rho(\tau)$ are precisely those which satisfy the extra conditions $D=A^\top$ and $C=B^\top$. Then $\rho$ gives an embedding $D_n\hookrightarrow \Aut(\Phi^*\Phi)$, with transitive action on the set $\{0,\ldots,2n-1\}$, and Lemma \ref{lem:red_aut} implies that $\Phi$ is strictly dihedral.\\

If $\Phi^*\Phi=[A,B;B^\top,A^\top]$ with $A,B$ negacirculant, we may define a projective monomial representation $\rho':D_n\to U(2n)$ by letting $\rho'(\mu)=\negc([0,0,\ldots,-1])\oplus\negc([0,0,\ldots,-1])$ and $\rho'(\tau)=\diag(1,-I_{n-1},1,-I_{n-1})\rho(\tau)$. We extend the definition of $\rho'$ to the entire dihedral group by $\rho'(\mu^j):=\rho'(\mu)^j$ and $\rho'(\mu^j\tau):=\rho'(\mu)^j\rho'(\tau)$, $0\le j\le n-1$. Then letting $M=\rho'(\mu)$ and $T=\rho'(\tau)$, it can be checked that $M^n=-I$, $T^2=I$ and $TM=M^{-1}T$. From this we infer that $\rho'$ is a projective representation. 
The centralizer of $\rho'(\mu)$ is the set of all matrices of the kind $[A,B;C,D]$ with $A,B,C,D$ negacirculant. Among these, the ones that centralize $\rho'(\tau)$ are those with $D=A^\top$ and $C=B^\top$. As before we apply Lemma \ref{lem:red_aut} to prove that $\Phi$ is projective dihedral.
\end{proof}

The following result is an immediate consequence of Theorem~\ref{thm:circpart}.

\begin{proposition}\label{pro:pdn}
    \begin{itemize}
        \item[]
        \item[(a)] The collection $\mathcal{SD}_{n}$ of all complex $2n\times 2n$ matrices of the form $[A,B;B^\top,A^\top]$ with $A,B$ complex circulant matrices, is a matrix $*$-algebra, i.e. closed under matrix addition, multiplication and conjugate-transpose. 
        \item[(b)] The collection $\mathcal {PD}_{n}$ of all complex $2n\times 2n$ matrices of the form $[A,B;B^\top,A^\top]$ with with $A,B$ negacirculant matrices, is a matrix $*$-algebra.
    \end{itemize}
\end{proposition}

\begin{proof} This can be checked directly. Alternatively,
part (a) (resp. (b)) follows from the fact that this set is the centralizing algebra of the representation $\rho$ (resp. $\rho'$) in the proof of the previous theorem.
\end{proof} 

A well-known fact about finite dimensional complex $*$-algebras is that they are \emph{semisimple} (see e.g. \cite{kenneth1996c}). An algebra is by definition semisimple, if it is a direct sum of minimal (left) ideals. The Artin-Wedderburn theorem then says that such finite dimensional algebras are isomorphic to a product of full matrix algebras over $\CC$. We shall compute the decompositions of $\mathcal{SD}_{n}$ and $\mathcal {PD}_{n}$ explicitly in Section \ref{sec:decomp} below. For more examples and a detailed discussion on Artin-Wedderburn decompositions of partitioned matrix algebras, see \cite[\S 3]{BENAV2024113908}.

\subsection{Decompositions of $\mathcal{SD}_{n}$ and $\mathcal {PD}_{n}$}\label{sec:decomp}

\begin{dfn}
    Let $\zeta$ be an $n$th root of unity. We define the \emph{cyclotomic} idempotent to be the circulant matrix
    $$E_\zeta\ := \ \tfrac{1}{n}\left(\zeta^{j-i} \right)_{i,j}.$$
    Similary,  let $\zeta$ be an $n$th root of $-1$. We define the \emph{nega-cyclotomic idempotent} to be the negacirculant matrix $$N_\zeta:=\tfrac{1}{n}(\zeta^{j-i})_{i,j}.$$
\end{dfn}
In the following, let $\mathcal C_n$ (respectively $\mathcal N_n$ ) denote the $*$-algebra of all complex circulant (respectively negacirculant) matrices. The next lemma shows that the idempotents $E_\zeta$ and $N_\zeta$ are responsible for the decomposition of these algebras.

\begin{lemma}\label{lem:circ_decomp}
    \begin{itemize}
        \item[]
        \item[(a)] There is an isomorphism of algebras $$ \mathcal C_n \cong \CC^n=\bigoplus_{\zeta\in \mu_n}\CC ,$$ where the map $\CC^n\to \mathcal C_n$ is given by $c\mapsto \sum\nolimits_{ \zeta\in \mu_n} c_\zeta E_{\zeta}$. The inverse map $\mathcal C_n\to \CC^n$ is given by $A\mapsto (a_\zeta(A))_\zeta$ where $a_\zeta(A)$ is defined by the equation $AE_\zeta=a_\zeta(A)E_\zeta$. The $*$ operator on $\mathcal C_n$ corresponds via this isomorphism to the complex conjugation on $\CC^n$.
        \item[(b)] There is an isomorphism of algebras $$ \mathcal N_n \cong \CC^n=\bigoplus_{\zeta\in \mu_{2n}\setminus\mu_n}\CC,$$ where the map $\CC^n\to \mathcal N_n$ is given by $c\mapsto \sum\nolimits_{\zeta\in \mu_{2n}\setminus \mu_n} c_\zeta N_{\zeta}$.  The inverse  map $\mathcal N_n\to \CC^n$ is given by $A\mapsto (a_\zeta(A))_\zeta$ where $a_\zeta(A)$ is defined by the equation $AN_\zeta=a_\zeta(A)N_\zeta$. The $*$ operator on $\mathcal N_n$ corresponds via this isomorphism to the complex conjugation on $\CC^n$.
    \end{itemize}
\end{lemma}

\begin{proof}
    We will prove (a), the proof of (b) being similar. The idempotents $E_{\zeta^i}$, $0\le i\le n-1$ are mutually orthogonal and self-adjoint. Therefore they are linearly independent, and by dimension counting, the map $\CC^n\to \mathcal C_n$ is a bijection. This map also respects multiplication on both sides from the orthogonality of the idempotents. It is easy to verify the formula for the inverse isomorphism, by multiplying the image of the first isomorphism by $E_\zeta$. Finally, the $*$-operator corresponds to complex conjugation on $\CC^n$ since the $E_\zeta$ are self-adjoint.
\end{proof}
Note that the isomorphism in (a) is just the discrete Fourier transform. We also note that the algebras $\mathcal C_n$ and $\mathcal{N}_n$ are isomorphic as $*$-algebras, as both are isomorphic to $\CC^n$. It is possible to give an algebra isomorphism directly, by sending the generator $\cir([0,0,\ldots,1])\mapsto \negc([0,0,\ldots, \zeta])$ where $\zeta^n=-1$. However, this choice is not an isomorphism of $*$-algebras.  We next prove similar results for $\mathcal{SD}_n$ and $\mathcal{PD}_n$.
Recall that $M_r(\mathbb{F})$ is the space of $r\times r$ matrices over a field $\mathbb{F}$.

\begin{proposition}
    \begin{itemize}
        \item[]
        \item[(a)] For any unordered pair of conjugate $n$th roots of $1$, $\{\zeta,\bar\zeta\}$, $\zeta\notin\mathbb R$, consider the idempotents $\mathcal E_{\zeta,\bar\zeta}:=\diag(E_\zeta+E_{\bar\zeta},E_\zeta+E_{\bar\zeta})\in \mathcal{SD}_n$. For $\zeta\in \{-1,1\}$, let $\mathcal E_\zeta:=\diag(E_\zeta,E_\zeta)$. Then the system $\{\mathcal E_{\zeta,\bar\zeta}\}\cup\{\mathcal E_{\pm 1}\}$ is a system of central self-adjoint and mutually orthogonal idempotents, and there is a $\CC$-algebra isomorphism 
        \be\label{eq:isom_dn} \mathcal{SD}_{n}\cong \begin{cases} \CC^2 \times M_2(\CC)^{(n-1)/2}, & n \text{ odd}\\ \CC^4\times M_2(\CC)^{n/2-1}, & n \text{ even}\end{cases}.\ee 
        Under this isomorphism, the $*$-operator becomes the conjugate-transpose on each component.
        \item[(b)] For any unordered pair of conjugate $n$th roots of $-1$, $\{\zeta,\bar\zeta\}$, $\zeta\notin\mathbb R$, consider the idempotents $\mathcal N_{\zeta,\bar\zeta}:=\diag(N_\zeta+N_{\bar\zeta},N_\zeta+N_{\bar\zeta})\in \mathcal{PD}_n$. For $\zeta=-1$, let $\mathcal N_{-1}:=\diag(N_{-1},N_{-1})$. Then the system $\{\mathcal N_{\zeta,\bar\zeta}\}\cup\{\mathcal N_{-1}\}$ is a system of central self-adjoint and mutually orthogonal idempotents, and there is a $\CC$-algebra isomorphism  
        \be \label{eq:isom_pdn} \mathcal {PD}_{n}\cong \begin{cases} \CC^2 \times M_2(\CC)^{(n-1)/2}, & n \text{ odd}\\ M_2(\CC)^{n/2}, & n \text{ even}\end{cases}.\ee
         Under this isomorphism, the $*$-operator becomes the conjugate-transpose on each component.
    \end{itemize}
\end{proposition}

\begin{proof}
We begin with (a). Note that for any circulant matrix $A\in \mathcal C_n$, $AE_\zeta=aE_\zeta$ for some scalar $a=a_\zeta(A)$, where $a$ is the $\zeta$ component of the corresponding vector in $\CC^n$ given by the isomorphism in part (a) of Lemma \ref{lem:circ_decomp}. Let $X=[A,B;B^\top,A^\top]$ be a general element in $\mathcal{SD}_n$. Then  \be\label{eq:artwed_dn} \mathcal E_{\zeta,\bar\zeta}X=\begin{bmatrix} aE_\zeta+bE_{\bar \zeta} & cE_\zeta+dE_{\bar \zeta}\\  dE_\zeta+cE_{\bar \zeta} & bE_\zeta+aE_{\bar \zeta}\end{bmatrix},\ee for some complex numbers $a,b,c,d$. On the other hand, for every choice of complex $a,b,c,d$, the matrix on the right remains the same after multiplication with 
the idempotent. Therefore the variables $a,b,c,d$ are independent, and 
the ideal $\mathcal E_{\zeta,\bar\zeta}\mathcal{SD}_n$ is 4-dimensional algebra over $\CC$. There is an algebra homomorphism $\mathcal E_{\zeta,\bar\zeta}\mathcal{SD}_n\to M_2(\CC)$, obtained by projecting the matrix in \eqref{eq:artwed_dn} to the component of $E_\zeta$, which is $[a,c;d,b]$. By dimension counting this is an isomorphism. Note also that the $*$-involution on $\mathcal E_{\zeta,\bar\zeta}\mathcal{SD}_n$ conjugates $a,b,c,d$ and interchanges between $c,d$. Thus it corresponds to the conjugate-transpose in $M_2(\CC)$. Now, \be\label{eq:compE1} \mathcal E_1X=\begin{bmatrix}
    aE_1 & bE_1\\ bE_1 & aE_1
\end{bmatrix},\ee and sending this matrix to the vector $(a+b,a-b)\in \CC^2$ becomes an algebra isomorphism $\mathcal E_1\mathcal{SD}_n\to \CC^2$. If $n$ is even, the same thing verbatim works with $E_{-1}$. Notice that also here the $*$-involution becomes complex conjugation. The self-adjoint idempotents $\mathcal E_{\zeta,\bar\zeta}$, $\mathcal E_1$ and $\mathcal E_{-1}$ (if $n$ is even) sum up to the identity $I_{2n}$, hence there is a $*$-algebra isomorphism $$\mathcal{SD}_n\ \cong \bigoplus E_{\zeta,\bar\zeta}\mathcal{SD}_n \oplus \mathcal E_1\mathcal{SD}_n \ (\oplus \mathcal E_{-1}\mathcal{SD}_n), $$ which proves the isomorphism in part (a).\\

Part (b) is proved similarly. The only difference is that when $n$ is odd we have the idempotent $\mathcal N_{-1}$ which is responsible for the $\CC^2$, and is missing when $n$ is even, hence the only components are $M_2(\CC)$. Although the dimensions are the same, there is no longer an isomorphism between $\mathcal{SD}_n$ and $\mathcal{PD}_n$ when $n$ is even.
\end{proof}

\subsection{Tight dihedral frames of redundancy $2$ in $\CC^n$}\label{sec:tightd}

In this section, we  classify the dihedral $TF(2n,n)$. Up to switching equivalence, their scaled Gram matrix $\tfrac{1}{2}\Phi^*\Phi$ is a rank $n$ idempotent, and can be brought to the form as in equation \eqref{eq:circpart}. By Theorem \ref{thm:circpart}, all we have to do is to classify the rank $n$ self-adjoint idempotents of the type $[A,B;B^\top,A^\top]$ with $A,B$ both (nega-)circulant. 

The goal now is to use the isomorphisms \eqref{eq:isom_dn}-\eqref{eq:isom_pdn} to construct all the self-adjoint idempotents of those algebras. Any self-adjoint idempotent of rank $n$ in $\CC^r\times M_2(\CC)^s$ can be expressed in a unique way as a sum of self-adjoint idempotents in each component. For a field component $\CC$, the only options are $0$ or $1$. For a matrix component $M_2(\CC)$, we have three options. Besides $0$ and $I_2$, there are idempotents of rank $1$. These are orthogonal projectors onto a 1-dimensional subspace. More explicitly, for every unit vector $u\in \CC^2$, the  idempotent $e_{u}=u^*u$ projects onto $\text{span}\{u\}$.  The self-adjoint idempotents of rank $n$ are obtained by collecting idempotents from this list, with total rank $n$. In what follows, we use $A\otimes B$ to denote the Kronecker product of matrices. This is the matrix with block structure $(A_{i,j}\cdot B)_{i,j}$. 

\begin{thm}[Strict Case]\label{thm:str_g} Let $\mu_n=H\sqcup F \sqcup C$ be a partition of the set $\mu_n$ of all 
$n^{th}$ roots of unity
such that $|F|=|C|$ and such that all sets $H,F,C$ are closed under complex conjugation. For each unordered conjugate pair $\{\zeta,\bar\zeta\}$, $\zeta\in H$ satisfying $\zeta^2\neq 1$, pick unit vectors $(v_\zeta,u_\zeta),(v_{\bar\zeta},u_{\bar\zeta})\in \CC^2$ such that $v_{\bar\zeta}=u_\zeta$ and $u_{\bar\zeta}=v_\zeta$. For each $\zeta\in H$ satisfying $\zeta^2=1$, let $(u_\zeta,v_\zeta)=(1/\sqrt{2},\pm 1/\sqrt{2})$. Then the $2n\times 2n$ matrix 
  \begin{equation}\label{equ:GramS} 
  X=\sum_{\zeta\in H} \begin{bmatrix}
        u_\zeta\overline{u_\zeta} & u_\zeta \overline{v_\zeta}\\ 
        v_\zeta \overline{u_\zeta} & v_\zeta\overline{v_\zeta}
    \end{bmatrix}\otimes E_\zeta +\sum_{\zeta\in F} I_2\otimes E_\zeta
   \end{equation} is a self-adjoint idempotent\footnote{This is the same as an orthogonal projection.} of $\mathcal{SD}_n$ of rank $n$. Furthermore, all self-adjoint idempotents of rank $n$ in $\mathcal{SD}_n$ are of this form.
    
\end{thm}

\begin{proof}
For $\zeta^2\neq 1$, every rank $1$ self-adjoint idempotent subordinate to $\mathcal E_{\zeta,\bar\zeta}$ is of the form $R_{u,v,\zeta}:=[u\bar u, u\bar v;v\bar u,v\bar v]\otimes E_\zeta+[\bar u u, \bar u v;u\bar v,\bar v v]\otimes E_{\bar\zeta}$ for a unit vector $(u,v)\in \CC^2$ (and this vector is uniquely determined up to a unimodular scalar, cf. \eqref{eq:artwed_dn} and the discussion afterwards). The rank $2$ idempotent is $S_\zeta:=I_2\otimes (E_\zeta+E_{\bar\zeta})$. Similarly, for $\zeta^2=1$, there are two rank $1$ idempotents, $R_{1,\zeta}:=\tfrac{1}{2}[1,1;1,1]\otimes E_\zeta$ and $R_{2,\zeta}=\tfrac{1}{2}[1,-1;-1,1]\otimes E_\zeta$ (cf. \eqref{eq:compE1}). The rank $2$ idempotent is $S_\zeta:=I_2\otimes E_\zeta=R_{1,\zeta}+R_{2,\zeta}$. Every rank $n$ self-adjoint idempotent is in a unique way a sum of distinct idempotents from this list, one per each component.

Let $H\subseteq\mu_n$ be a union of those of $\{\zeta,\bar\zeta\}$, from which we take idempotents of rank $1$. Let $F\subseteq \mu_n$ be a union of those of $\{\zeta,\bar\zeta\}$, from which we take idempotents of rank $2$. Let $C\subseteq\mu_n$ be the complement of $H\cup F$. The rank of the total idempotent is $|H|+2|F|=n$, while total cardinality $|H|+|F|+|C|=|\mu_n|=n$. We conclude that $|F|=|C|$. For each $\{\zeta,\bar\zeta\}\subset H$ we choose a unit vector $(u_\zeta,v_\zeta)\in \CC^2$ and the idempotent $R_{u_\zeta,v_\zeta,\zeta}$. Likewise, if $\zeta^2=1$ and $\zeta\in H$, we choose the idempotent $R_{\ell,\zeta}$, $\ell=1$ or $2$. For $\zeta\in F$ we choose $S_\zeta$. The total idempotent is $$X=\sum_{\{\zeta,\bar\zeta\}\subseteq H,\zeta^2\neq 1} R_{u_\zeta,v_\zeta,\zeta} +\sum_{\{\zeta,\bar\zeta\}\subseteq H,\zeta^2= 1} R_{i_\zeta,\zeta}+\sum_{\{\zeta,\bar\zeta\}\subseteq F} S_{\zeta}.$$ This sum equals the one in \eqref{equ:GramS}, which concludes the proof that every self-adjoint rank $n$ idempotent in $\mathcal{SD}_n$ is of this form. The converse also holds, as it can be seen directly that $X$ in \eqref{equ:GramS} is such an idempotent. 
\end{proof}

Similarly, in the projective case, we have:

\begin{thm}[Projective Case]\label{thm:proj_g} Let $\mu_{2n}\setminus \mu_n=H\sqcup F \sqcup C$ be a partition such that $|F|=|C|$ and such that all sets $H,F,C$ are closed under complex conjugation. For each unordered conjugate pair $\{\zeta,\bar\zeta\}$, $\zeta\in H$ satisfying $\zeta\neq -1$, pick unit vectors $(v_\zeta,u_\zeta),(v_{\bar\zeta},u_{\bar\zeta})\in \CC^2$ such that $v_{\bar\zeta}=u_\zeta$ and $u_{\bar\zeta}=v_\zeta$.  If $\zeta\in H$ satisfies $\zeta=-1$, let $(u_\zeta,v_\zeta)=(1/\sqrt{2},\pm 1/\sqrt{2})$. Then the $2n\times 2n$ matrix 
  \begin{equation}\label{equ:GramP}
   X=\sum_{\zeta\in H} \begin{bmatrix}
        u_\zeta\overline{u_\zeta} & u_\zeta \overline{v_\zeta}\\ 
        v_\zeta \overline{u_\zeta} & v_\zeta\overline{v_\zeta}
    \end{bmatrix}\otimes N_\zeta +\sum_{\zeta\in F} I_2\otimes N_\zeta
  \end{equation}
   is a self-adjoint idempotent of $\mathcal {PD}_n$ of rank $n$. Furthermore, all self-adjoint idempotents of rank $n$ in $\mathcal {PD}_n$ are of this form.
\end{thm}
\begin{proof}
    Similar to the proof of Theorem \ref{thm:str_g}.
\end{proof}

\begin{remark}
These two theorems claim that $SD_nv$ (resp. $PD_nv$)$\in \tf{2n}{n}$ if and only if, up to switching equivalence, its scaled Gram matrix $\tfrac12\Phi^*\Phi$ is of the form \eqref{equ:GramS} (resp. \eqref{equ:GramP}).
\end{remark}

\begin{remark}
    The results here should be compared to \cite[Theorem 10.8]{Waldronbook}, which gives a general criteria for a $G$-frame to be tight. The description there seems very different at first sight. We may view the TF $\Phi$ as an orthogonal projection from the strict (projective) representation onto a subrepresentation. The description there is in terms of the embedding matrix $\Phi$ of the subrepresentation into the strict representation, written in a basis subordinate to the irreducible representational decomposition of the former. Our description is of the projection operator in terms of the (nega)circulant basis. In our case, each constituent appears in the strict representation with multiplicity 2, and we are choosing whether to leave it as is ($F$), reduce the multiplicity to $1$ ($H$) or to $0$ ($C$).
\end{remark}

\subsection{Regular dihedral configurations and ETFs}\label{sec:regular}
This section builds upon previous results of dihedral configurations, and further characterizes tight or equiangular tight dihedral frames assuming it is regular.
\begin{dfn}
    Let $\Phi=\rho(D_n)v$ be an $n\times 2n$ dihedral configuration, where $\rho$ is a strict or projective representation of the dihedral group $D_n=\{1,\mu,\ldots,\mu^{n-1},\tau,\mu\tau,\ldots,\mu^{n-1}\tau\}$. We say that $\Phi$ is \emph{regular} if the subset $\{v,\rho(\mu)v,\ldots, \rho(\mu^{n-1})v\}$ is a basis for $\CC^n$.\label{def:regular}
\end{dfn}

This condition is equivalent to the requirement that in the Gram matrix $[A,B;B^\top,A^\top]$, the matrix $A$ is invertible. We shall then say that this Gram matrix is \emph{regular}. Applying $\rho(\tau)$ to the basis $\{v,\rho(\mu)v,\ldots, \rho(\mu^{n-1})v\}$, we see that the reordered set $\{\rho(\tau)v,\rho(\mu\tau)v,\ldots,\rho(\mu^{n-1}\tau)v\}$ is also a basis, showing that $B$ is invertible too. Conversely, if $B$ is invertible, then so must be $A$, so $B$ being invertible is also equivalent to regularity. Even though regularity looks like a generic condition, it is not a necessary condition for a dihedral configuration to be a frame. We can even have tight frames that are not regular.


\begin{ex}
Let $\Phi$ be defined as in \eqref{equ:Dnv}, which is a strict dihedral configuration. Using \eqref{equ:dft}, we have for $v=\Fc x$
$$\rank([v, Mv, \cdots, M^{n-1}v])= \rank([x, \Tc x, \cdots, \Tc^{n-1}x]) = \rank(\diag(\Fc x))=|\text{supp}(v)|.$$
So given $v=[1,\ldots,1,0]^\top$, $\Phi$ will not be regular as the first $n$ vectors only span a subspace of dimension $n-1$, but one can check that $\Phi$ spans $\C^n$, and thus is a frame.
\end{ex}

\begin{ex}
    We can use Theorem \ref{thm:proj_g} to generate tight dihedral frames that are not regular. Suppose that $n$ is even and let $F\subset \mu_n$ be a set of cardinality $n/2$. Then $H$ is empty and the idempotent $X=I_2\otimes \sum_{\zeta\in F}  N_\zeta$ gives rise to a $n\times 2n$ strict dihedral tight frame. However, $B=0$ in that matrix and $\rank (A)=n/2$, so the frame is not regular. 
\end{ex}

The next proposition characterizes regular dihedral tight frames.

\begin{proposition}\label{prop:reg}
    Let $H,F,C$ be as in Theorems \ref{thm:str_g} or \ref{thm:proj_g}. Then the idempotent $X$ in \eqref{equ:GramS} or \eqref{equ:GramP} is regular if and only if $|H|=n$, and for all $\zeta\in H$, $u_\zeta\overline{v_\zeta}\neq 0$. 
\end{proposition}
\begin{proof}
    Writing $X=[A,B;B^\top,A^\top]$, the rank of $B$ equals the cardinality of $H$, this follows immediately from the formulas for the idempotent $X$, and from the fact that regularity is equivalent to $\rank(B)=n$.
\end{proof}

\begin{proposition}\label{prop:real} Let $H,F,C$ be as in Theorems \ref{thm:str_g} or \ref{thm:proj_g}.
    If $X=[A,B;B^\top,A^\top]\in \mathcal{SD}_n$ or $\mathcal{PD}_n$ is a self-adjoint idempotent of rank $n$, then in the strict case $$Re(A)=\tfrac{1}{2}\sum_{\zeta\in H} E_\zeta+\sum_{\zeta\in F} E_\zeta,$$ and in the projective case
    $$Re(A)=\tfrac{1}{2}\sum_{\zeta\in H} N_\zeta+\sum_{\zeta\in F} N_\zeta.$$
\end{proposition}

\begin{proof} Let us prove this in the projective case. The strict case is similar.
    We have $A=\sum_{\zeta\in H} u_\zeta\overline{u_\zeta} N_\zeta+\sum_{\zeta\in F} N_\zeta$. Hence $A_{i,j}=\tfrac{1}{n}\sum_{\zeta\in H} u_\zeta\overline{u_\zeta} \zeta^{j-i}+\tfrac{1}{n}\sum_{\zeta\in F} \zeta^{j-i}$. Using the equations $u_{\bar\zeta}=v_\zeta$, $u_{\bar\zeta}=v_\zeta$ and the fact that $H$ and $F$ are closed under complex conjugation, we may rewrite $A_{i,j}=\tfrac{1}{n}\sum_{\zeta\in H} v_\zeta\overline{v_\zeta}\zeta^{i-j}+\tfrac{1}{n}\sum_{\zeta\in F} \zeta^{i-j}$. Combining both equations we compute  $$A_{i,j}+\overline{A_{i,j}}= \tfrac{1}{n}\sum_{\zeta\in H} (u_\zeta\overline{u_\zeta}+v_\zeta\overline{v_\zeta})\zeta^{j-i}+ \tfrac{1}{n}\sum_{\zeta\in F}2 \zeta^{j-i}=\sum_{\zeta\in H} (N_{\zeta})_{i,j}+2\sum_{\zeta\in F} (N_{\zeta})_{i,j}.$$
\end{proof}

One important corollary is the finiteness theorem:

\begin{thm}\label{thm:finite}
    For a fixed natural number $n$, up to switching equivalence, there are only finitely many dihedral $ETF(2n,n)$.
\end{thm}

\begin{proof}
    Let $\Phi^*\Phi$ be the Gram of a dihedral $ETF(2n,n)$. Up to switching equivalence, $\tfrac{1}{2}\Phi^*\Phi=[A,B;B^\top,A^\top]$ is a rank $n$ self-adjoint idempotent, and $B$ is a real matrix with entries $\pm 1/2\sqrt{2n-1}$. By Proposition \ref{prop:real} there are finitely many options for $Re(A)$. Since $|A_{i,j}|$ are known, there are finitely many possible values of $A$. Altogether, there are finitely many options for $\Phi^*\Phi$, up to switching equivalence. 
\end{proof}
Another corollary of Proposition \ref{prop:real} is
\begin{cor}\label{cor:H=n}
    Let $X=[A,B;B^\top,A^\top]\in \mathcal{SD}_n$ or $\mathcal{PD}_n$ be a self-adjoint idempotent of rank $n$. Then all off diagonal entries of $A$ are purely imaginary, iff $|H|=n$.
\end{cor}
\begin{proof}
    If $|H|=n$ this is clear by the proposition. If $|H|<n$, $Re(A)$ must have a nonzero off-diagonal entry, since $\diag(A)=\tfrac{1}{2}I$ and the idempotent systems $\{E_\zeta\}$ and $\{N_\zeta\}$ are linearly independent.  
\end{proof}
From this we derive the following corollary:

\begin{cor}\label{prop:imag}
    Let $X=[A,B;B^\top,A^\top]\in \mathcal{SD}_n$ or $\mathcal{PD}_n$ be a self-adjoint idempotent of rank $n$. If $X$ is regular, then $A$ is imaginary off the main diagonal. 
\end{cor}

\begin{proof}
If $X$ is regular, then $|H|=n$ and $F=\emptyset$, and $Re(A)=\tfrac{1}{2}\sum_{\zeta} N_\zeta$ (or $\tfrac{1}{2}\sum_\zeta E_\zeta$) $=\tfrac{1}{2}I_n$.    
\end{proof}

The converse of this claim is not true for general tight frames, since for regularity we need to satisfy the condition $u_\zeta \overline{v_\zeta}\neq 0$ for all $\zeta\in H$. However, we will show below that it is true ETFs.\\


The construction in Lemma \ref{lem:shm} is not suitable for the type of symmetries that we study in this paper. The next lemma gives a slight variant.
\newcommand{\mH}{\mathcal H}
\begin{lemma}\label{lem:newshm}
    Let $\mH$ be real skew Hadamard of size $2n$, write $\mH=[P,Q;-Q^\top,R]$ for  $P,Q,R\in\{-1,1\}^{n\times n}$. Then 
    \be\label{eq:shm2etf}
    M=M(\mH) := I_{2n}+\frac{1}{\sqrt{2n-1}}\begin{bmatrix}
        \i(P-I_n) & Q\\ Q^\top & \i(R-I_n)
    \end{bmatrix}
    \ee is the Gram matrix of an $\etfn$.
\end{lemma}

\begin{proof}
The matrix $\mH-I_{2n}=[P-I_n, Q; -Q^\top, R-I_n]$ is a skew symmetric conference matrix, so it has 0's on the main diagonal. Then so is true for $M-I_{2n}$. Since the off diagonal entries of $[P-I_n, Q; -Q^\top, R-I_n]$ are $\pm 1$, we then have the off diagonal entries of $M$ to have modulus $\frac{1}{\sqrt{2n-1}}$ following the Welch bound \eqref{equ:welch}.

    The matrix $M$ is self-adjoint since $\mH+\mH^\top = 2I$. 
   Since $\mH$ is skew Hadamard, it follows that $\tfrac{1}{2}M$ is an idempotent of rank $n$. The details are similar to the proof of Lemma \ref{lem:shm}.
   \end{proof}

\begin{remark}\label{rem1}
    \begin{itemize}
        \item[(a)] The constructions of Lemma \ref{lem:newshm} and Lemma \ref{lem:shm} are switching equivalent. The Gram matrices are diagonally similar by $\Pi = \diag(I_n,\i I_n)$. 
        \item[(b)] By Lemma \ref{lem:shm} and comment (a), any Gram matrix of an $\etfn$ whose entries are purely imaginary outside the main diagonal is diagonally similar to an $M(\mH)$ for a skew Hadamard $\mH$. 
        \item[(c)] Automorphisms of $\mH$ by monomial conjugations carry over automorphisms of $M(\mH)$.
        \item[(d)] Note that the matrix $M(\mH)$ depends on the chosen partition for $\mH$. But from (a) all choices lead to switching equivalent $M(\mH)$. 
    \end{itemize}
\end{remark}

In view of  Lemma \ref{lem:newshm} we have the following characterization of regular dihedral $\etfn$. 

\begin{thm}\label{thm:main}
Let $\Phi=SD_nv$ (resp. $\Phi=PD_nv$)  be a dihedral configuration. Then $\Phi$ is a regular dihedral $\etfn$ if and only if it is switching equivalent to a configuration $\Phi'$ whose Gram can be written as 
    \be\label{eq:MH} \Phi'^*\Phi'=\ I_{2n}+\frac{1}{\sqrt{2n-1}}\begin{bmatrix}
        \i(P-I_n) & Q\\ Q^\top & \i(P^\top-I_n)\end{bmatrix}.\ee
        for a real skew Hadamard matrix $\mH=[P,Q;-Q^\top,P^\top]$ with $P, Q$ circulant (resp. negacirculant).
\end{thm}
\begin{proof}
Suppose first that $\Phi$ is regular dihedral. By Theorem \ref{thm:circpart} we may replacing $\Phi$ by a switching equivalent $\Phi'$ such that 
$$\Phi'^*\Phi'=G_2=\begin{bmatrix}A&B\\B^\top&A^\top\end{bmatrix}:=\ I_{2n}+\frac{1}{\sqrt{2n-1}}\begin{bmatrix}
        \i(P-I_n) & Q\\ Q^\top & \i(P^\top-I_n)\end{bmatrix},$$ for $A,B$ circulant (resp. negacirculant) and $P,Q$ are defined by the formula.
By direct calculation, we have $P-I=-\i\sqrt{2n-1}(A-I), Q = \sqrt{2n-1}B$. We know that $Q$ is real by Theorem \ref{thm:circpart}. Since $\Phi$ is also regular, then by Corollary \ref{prop:imag}, $P$ is real too. Let $$G_1=I_{2n}+\frac{\i}{\sqrt{2n-1}}\begin{bmatrix}P-I&Q\\-Q^\top&P^\top-I\end{bmatrix}=: I+\frac{\i}{\sqrt{2n-1}}C.$$
Remark \ref{rem1}(a) states that
$\Pi G_1\Pi^*=G_2$ where $\Pi = \diag(I_n, iI_n)$, and they are switching equivalent. Thus $G_1$ is the Gram matrix of an ETF. Since $C$ is real, Lemma \ref{lem:shm} shows that $\mH=C+I$ is skew Hadamard. Moreover, either $P,Q$ are both circulant or both negacirculant. The first option is for strict dihedral ETF and the second option is for genuine projective dihedral ETF.\\
 

Conversely, suppose that $\mH=[P,Q,-Q^\top,P^\top]$ is real skew Hadamard with $P,Q$ both circulant or negacirculant. Lemma \ref{lem:newshm} tells us that \eqref{eq:MH} is the Gram matrix of an $\etfn$. We need to show that this ETF is regular. Let us show this for the negacirculant case. The proof in circulant case is similar, or alternatively, we will see soon below that the circulant case is vacuous. The rank $n$ idempotent $X:=\tfrac{1}{2}G_2$ has the formula \eqref{equ:GramP} with respect to a partition $\mu_{2n}\setminus \mu_n=H\sqcup F\sqcup C$, and by Corollary \ref{cor:H=n} we must have $|H|=n$. Suppose by contradiction that our ETF is not regular, which means that for some $\zeta=\exp(2\mathrm i k\pi/2n)\in \mu_{2n}\setminus \mu_n$ ($k$ odd), $u_\zeta\overline{v_\zeta}=0$ (Proposition \ref{prop:reg}). The orthogonality of the conference matrix implies that $-(P-I)^2+QQ^\top=(P-I)(P-I)^\top+QQ^\top=(2n-1)I_n$. We project this equality to the $\zeta$-component of $\mathcal{PD}_n$. The condition $u_\zeta \overline{v_\zeta}=0$ implies that the projection of $QQ^\top$ is zero (see \eqref{equ:GramP}). Hence we obtain the scalar equation $$-\left(\sum_{m=1}^n e_m\zeta^m\right)^2=2n-1$$ where $(0,e_1,\ldots,e_n)$ is the first row of $P-I$. This implies in particular that $\sqrt{1-2n}\in \mathbb Q(\exp(2\mathrm i \pi/2n))$. If $2n-1$ is not a perfect square, this contradicts the the well-known fact that all quadratic extensions of $\mathbb Q$ inside $\mathbb Q(\exp(2\mathrm i \pi/2n))$ are of the form $\mathbb Q(\sqrt{d})$ for $d$ that divides $2n$. Therefore $2n-1$ must be a perfect square. But then we will have $\sqrt{-1}\in \mathbb Q(\exp(2\mathrm i \pi/2n))$, which implies $n$ is even. But now $2n-1\equiv 3\mod 4$ cannot be a perfect square, which is a contradiction. 
\end{proof}

\begin{remark*}
    The theorem settles the converse claim of Corollary \ref{prop:imag} for dihedral ETFs. By Remark \ref{rem1}(b) we know that the Gram matrix that has purely imaginary off-diagonal entries is diagonally equivalent to an $M(\mH)$, $\mH$ skew-Hadamard. The theorem implies now that such ETFs are regular. 
\end{remark*}

Theorem \ref{thm:main} transfers the task of searching for dihedral $\etfn$ to searching for skew Hadamard matrix structured by $[P,Q,-Q^\top,P^\top]$ where $P, Q$ both circulant or both negacirculant. This is still a challenging task, but at least makes the search discrete. Let us prove now that such Hadamard matrices with $P, Q$ both circulant do not exist.

\begin{lemma}\label{lem:no}
    There is no skew Hadamard matrix that is also \bc. 
\end{lemma}

\begin{proof}
Let $C=(c_{ij})$ be circulant, which is determined by its first column $[c_{11}, c_{21}, \cdots, c_{n1}]^\top$.
We define a linear functional $S$ from the space of all $n\times n$ circulant matrices to $\C$ as $S(C)=\sum_{i=1}^nc_{i1}$.
It is easy to check that $S$ is a ring homomorphism: $S(CD)=S(C)S(D)$ and $S(C+D)=S(C)+S(D)$ for circulant matrices $C, D$. 
We can extend $S$ to be defined on \bc\  matrices $\mH=[A, B; C, D]$ by letting $S(\mH)=[S(A), S(B); S(C), S(D)]$. This defines again a ring homomorphism from the space of \bc\ matrices to $M_2(\C)$.\\

Now suppose there exists a \bc\ skew Hadamard matrix $\mH=[A,B; C, D]$. We must have $A-I=-(A-I)^\top$. Applying $S$ to this equation, we get $S(A-I)=-S(A-I)$ ($S$ respects the transpose). Likewise, $S(D-I)=0$.

Let $b=S(B), c=S(C)$. Now apply $S$ to $(\mH-I)(\mH-I)^\top = (2n-1)I$, we get 
\begin{align*}
&S(\mH-I)S(\mH^T-I)=(2n-1)I\\
\Longleftrightarrow&\begin{bmatrix}0&b\\c&0\end{bmatrix}\begin{bmatrix}0&c\\b&0\end{bmatrix}=\begin{bmatrix}2n-1&0\\0&2n-1\end{bmatrix}\\
\Longleftrightarrow& \ b^2=c^2=2n-1
\end{align*}
A skew Hadamard matrix requires its order $2n$ being a multiple of 4, so $b^2\equiv3 \text{ mod 4}$, which is impossible for an integer $b$.
\end{proof}

Another structural result about \bnc\ skew Hadamard matrices is the following.

\begin{thm}
    Let $\mH=[P,Q;-Q^\top,R^\top]$ be a \bnc\ skew Hadamard matrix. Then $Q$ is nonsingular, and $R=P$.
\end{thm}

\begin{proof}
    Clearly $\mH'=[P,Q;-Q^\top,P^\top]$ is a \bnc\ skew Hadamard matrix and $M(\mH')$ is the Gram of a regular dihedral $ETF(2n,n)$ as Theorem \ref{thm:main} tells us. In particular $Q$ is nonsingular, and this forces $R=P$.
\end{proof}

Combining Theorem \ref{thm:main} and Lemma \ref{lem:no} immediately gives the following result.

\begin{thm}\label{thm:2negcirc}
    Let $\Phi$ be a regular dihedral $\etfn$. Then $\Phi$ is genuinely projective, and there exists
    a \bnc\ skew Hadamard matrix $\mH=[P,Q;-Q^\top,P^\top]$ such that,
    up to switching equivalence, $\Phi^*\Phi=\mathcal G$, where
    $$\mathcal G:=I_{2n}+\frac{1}{\sqrt{2n-1}}\begin{bmatrix}
        \i(P-I_n) & Q\\ Q^\top & \i(-P+I_n)\end{bmatrix}.$$
    Conversely, for every such $\mH$, the matrix $\mathcal G$ is the Gram matrix of a regular dihedral $\etfn$.
\end{thm}
    
This result can be compared to our Theorem \ref{thm:abelian}, which implies the same thing for abelian $\etfn$. The dihedral case is considerably harder to prove, and follows from specific properties of the dihedral group. The following is an immediate corollary of the previous results given the restriction on Hadamard matrices.

\begin{cor}\label{cor:regular} Given a natural number $n$, the following statements hold. 
    \begin{itemize}
        \item[]
        \item[(a)] There are no regular dihedral $ETF(2n,n)$ for odd $n>1$.
        \item[(b)] There are no strict regular dihedral $ETF(2n,n)$.
    \end{itemize}   
\end{cor}

\section{Paley and Double Paley $ETF$s as dihedral frames}\label{sec5}

In this section, we introduce the notion of Paley ETFs and use Lemma \ref{lem:red_aut} to show that all Paley $\etfn$ for $n$ even are genuine projective dihedral frames. In addition, we prove that the  doubling of the Paley $\etfn$ results also in a projective dihedral $\etf{4n}{2n}$. However, doubling again does not bequeath this property.

\subsection{Paley ETFs} Let $q\equiv 3\mod 4$ be a prime power, and denote by $\FF_q$ the field of $q$ elements. Let $V=\FF_q^2$. The projective line over $\FF_q$ is the set $\PP^1_q:=(V\setminus \{0\})/\FF_q^\times$, of all nonzero vectors modulo proportionality. It satisfies $|\PP^1_q|=q+1$. We write each point in $\PP^1_q$ as $[a:b]$ where $(a,b)\neq (0,0)$ and where $[a:b]=[\lambda a:\lambda b]$ for all $\lambda\in \FF_q^\times$. Paley~\cite{Paley1933} constructed a skew Hadamard matrix $H_q$ of order $q+1$, from which we can construct an $\etf{q+1}{(q+1)/2}$ using Lemma \ref{lem:shm} (or Lemma \ref{lem:newshm}). This construction only defines the frame from the Gram matrix, hence it is only defined up to equivalence. An explicit construction of the vectors of this frame can be found in \cite{4373397,Xia2005AchievingTW,zaun1999}. 

The construction of the Paley Gram matrix is as follows. Let $\langle,\rangle$ be the determinant form on $V$, that is $\langle (a,b),(c,d)\rangle:=ad-bc$. For each point $P\in \PP^1_q$ let us choose a representative vector $\hat P \in V\setminus\{0\}$, such as the vectors $(a,1)$ for each $a\in \FF_q$ and $(1,0)$. Let $\leg{\cdot}{q}$ denote the \emph{generalized} Legendre symbol over $q$, which explicitly means
$$ \leg{a}{q}:=\begin{cases}
    1 & \exists\ x\in \FF_q \text{ such that } a=x^2\\
    -1 & \text{otherwise}
\end{cases}.$$
This symbol is multiplicative in $a$. When $q$ is prime and $a\neq 0$ this symbol agrees with the usual Legendre symbol. Notice the nonstandard choice for $a=0$.
Define the Paley Hadamard matrix as 
\begin{equation}
H_q\ := \left( \leg{\langle \hat P,\hat Q\rangle}{q}\right)_{P,Q\in \PP^1_q}.
\end{equation}
Then it has been proved by Paley that $H_q$ is a skew Hadamard matrix. The skewness is evident from the fact that $\langle,\rangle$ is an antisymmetric bilinear form and from the fact that $\leg{-1}{q}=-1$. The definition of $H_q$ depends on the ordering we choose on $\PP^1_q$, as well as on the liftings $\hat P$ for each $P$. Those ambiguities define $H_q$ up to conjugation by a real monomial matrix, and in turn the same for the Gram matrix $M(H_q)$ of the ETF, as defined in \eqref{eq:shm2etf}. Consequently, the Paley ETF is defined here only up to switching equivalence.\\

Our next step is to identify a large subgroup of switching automorphisms. Let $PGL(2,q)$ be the group of $2\times 2 $ matrices over $\FF_q$, modulo identifying matrices up to a scalar multiple. We denote by  $PSL(2,q)$ the subgroup of $PGL(2,q)$ consisting of matrices with determinant $1$,  modulo the normal subgroup of scalar matrices (which includes only $\pm I_2$). We note that   $[PGL(2,q):PSL(2,q)]=2$ and that $PSL(2,q)$ has order $|PSL(2,q)|=(q^3-q)/2$. The group $PGL(2,q)$ acts on $\PP^1_q$ by fractional linear transformations: $(a,b;c,d)[u:v]=[au+bv:cu+dv]$ in a well-defined manner. Each $g\in PGL(2,q)$ thus induces a permutation on $\PP^1_q$, and we denote the corresponding permutation matrix by $S(g)$. If we choose a lift of each $g$ to $\hat g\in GL(2,q)$, then $\hat g\hat P$ differs from $\widehat{gP}$ by a scalar $\lambda(g,P) \in\FF_q^\times$: $\hat g\hat P=  \lambda(g,P)\cdot\widehat{gP}$, and the multiplication by $\hat g$ defines a monomial matrix $\Pi(g):=S(g)\cdot D(g)$ where $D(g)=\diag\leg{\lambda(g,P)}{q}_P$.

 Given a $\{\pm 1\}$-Hadamard matrix $H$, we define the automorphism group $\Aut(H)$ as the group of all pairs $(L,R)$ of $\{\pm 1\}$-monomial matrices such that $LHR^\top=H$, modulo the identification of $(L,R)$ with $(-L,-R)$.
 The following result describes some group homomorphisms needed in the proof that the Paley-type ETFs are projective dihedral.

\begin{lemma}\label{lem:dihedprojrep}
    \begin{itemize}
        \item[(a)] The map $g\mapsto \Pi(g)$ is a projective monomial representation of $PGL(2,q)\to U(q+1)$. The induced map $PGL(2,q)\to \mon(q+1)/\{\pm I\}$ is a well-defined homomorphism.
        \item[(b)] For each $g\in PSL(2,q)$ we have $\Pi(g)H_q\Pi(g)^*=H_q$, which induces an embedding $PSL(2,q)\hookrightarrow \Aut(H_q)$.
        \item[(c)] For each $g\in PGL(2,q)\setminus PSL(2,q)$, $\Pi(g)H_q\Pi(g)^*=H_q^\top$.
    \end{itemize}
\end{lemma}

\begin{proof} For each $P\in \PP^1_q$ and $g_1,g_2\in PGL(2,q)$ we have 
    $$
        \lambda(g_2,g_1P)^{-1}\hat g_2 \lambda(g_1,P)^{-1}\hat g_1 \hat P= \lambda(g_2,g_1P)^{-1}\hat g_2 \widehat{g_1P}=\widehat{g_2g_1P}=\lambda(g_2g_1,P)^{-1}\widehat{g_2g_1}\hat P.
    $$
    Now, $\hat g_2\hat g_1=\chi \widehat{g_2g_1}$ for a scalar $\chi$ independent of $P$. Comparing scalars we obtain\\ $\lambda(g_2,g_1P) \lambda(g_1,P)= \chi \lambda(g_2g_1,P)$  for $\chi$ independent of $P$. Taking the generalized Legendre symbol of both sides
    we arrive at the identity $D(g_2)_{g_1P}D(g_1)_P=\pm D(g_2g_1)_P$, where the sign is independent on $P$. By applying the identity of diagonal matrices $(S(g_1)^{-1}D(g_2)S(g_1))_P=D(g_2)_{g_1P}$, this amounts further to the matrix identity $\Pi(g_2g_1)=\pm \Pi(g_2)\Pi(g_1)$, which proves (a).\\

    Now let $g\in PSL(2,q)$. The $(gP,gQ)$ entry of $\Pi(g)H_q\Pi(g)^*$ for $P\neq Q$ is \be \label{eq:det=1} \leg{\lambda(g,P)}{q}\leg{\lambda(g,Q)}{q}^{-1}\leg{\left\langle \widehat{P},\widehat{Q}\right\rangle}{q}=\leg{\lambda(g,P)}{q}\leg{\lambda(g,Q)}{q}^{-1}\leg{\left\langle \hat g\hat P,\hat g\hat Q\right\rangle}{q}=\leg{\left\langle \widehat{gP},\widehat{gQ}\right\rangle}{q},\ee
    which shows that $\Pi(g)H_q\Pi(g)^*=H_q$. In the first equality, we used $\det(\hat g)=1$ and in the second the definition of $\lambda$. This gives a map $PSL(2,q)\to \Aut(H_q)$. It remains to show the injectivity of this map. If $g\neq 1$ then $S(g)$ is the nonzero map on $\PP^1_q$ which shows that $\Pi(g)$ is not scalar, hence injectivity and the proof of (b) is complete.\\
    
    The proof of (c) is done by a similar chain of equalities as \eqref{eq:det=1}. The only difference is that now $\det(\hat g)=-1$ (or any other non-square), which negates the sign of the determinant form when $P\neq Q$, so we obtain on the right hand side $-\leg{\left\langle \widehat{gP},\widehat{gQ}\right\rangle}{q}$. Using the skewness of $H_q$, we prove (c).
\end{proof}

\begin{remark}
    While $PSL(2,q)$ embeds into $\Aut(H_q)$, it is not the full automorphism group. There are automorphisms coming from $PGL(2,q)$ that are not symmetric (meaning $L\neq R$ in the definition above). There are also automorphisms coming from the Galois action on $\FF_q$, and adjoining them generates the full automorphism group for $q>11$, see \cite{DELAUNEY20082910,Kantor1969AutomorphismGO}. The notion of an automorphism is slightly different there, but in our specific case, all automorphisms turn out to satisfy the definition here.
\end{remark}

Our next goal is to show that Paley ETFs are projective dihedral. We will accomplish this via Lemma \ref{lem:red_aut}, by finding an isomorphic copy of $D_n$ inside $\Aut(H_q)$. Let $\FF_{q^2}=\FF_q(\sqrt{-1})$ be the quadratic extension of $\FF_q$, also viewed as a 2-dimensional vector space over $\FF_q$ with basis $1,\sqrt{-1}$. The reader should be cautioned that we distinguish this modular $\sqrt{-1}$ from the complex $\i$. The multiplication operator by $a+b\sqrt{-1}\in \FF_{q^2}$ is represented with respect to this basis by the matrix $[a,-b;b,a]\in GL(2,q)$. This gives a group homomorphism $\FF_{q^2}^\times\to GL(2,q)$. There is a subgroup $S\subset \FF_{q^2}^\times$ consisting of all elements $a+b\sqrt{-1}$ such that $a^2+b^2=1$. This reduces to a map $S\to SL(2,q)$, and the only element in $S$ mapping to a scalar matrix is $-1$. Thus we have a well-defined homomorphism
$$ \beta: S/\langle -1 \rangle \to PSL(2,q), \ \ \ a+b\sqrt{-1}\mapsto \begin{bmatrix}
    a & -b\\b & a
\end{bmatrix}.$$
This homomorphism is easily seen to be an embedding. The cardinality of $S$ is $q+1$, as it is the kernel of the determinant (norm) map from $\FF_{q^2}^\times \to \FF_q^\times$. The group $S/\langle -1 \rangle$ has cardinality $(q+1)/2$. This gives us a copy of half of a dihedral subgroup in $PSL(2,q)$. Let $m\in PSL(2,q)$ be the image of a generator of $S$. Next, consider the element $t\in PSL(2,q)$ represented by the matrix $\mathbf t:=[\alpha,\gamma;\gamma,-\alpha]\in SL(2,q)$ such that $\alpha^2+\gamma^2=-1$ (in a finite field every element is the sum of 2 squares). Conjugating $[a,-b;b,a]$ by $\mathbf t$ results in $[a,-b;b,a]^{-1}$; in particular $tmt^{-1}=m^{-1}$. Thus we have the following result. 

\begin{lemma}\label{lem:mt}
 The order of the element $m\in PSL(2,q)$ is  $(q+1)/2$, while the element $t$ has order $2$. Furthermore, we have  $tmt^{-1}=m^{-1}$ in $PSL(2,q)$.
\end{lemma}

\begin{proof}
    The order of $m$ is clear from the injectivity of $\beta$. The other statements are easily checked.
\end{proof}
Therefore we see a copy of $D_n$, $n=(q+1)/2$ inside $PSL(2,q)$. Let us denote this subgroup by $\Delta_n:=\langle m,t\rangle$. 
\begin{remark}
    Let $\mathbf m,\mathbf t\in SL(2,q)$ be matrices above $m,t$. They satisfy $\mathbf m^n=\mathbf t^2=-I_2$ and $\mathbf{tmt}^{-1}=\mathbf{m}^{-1}$.
\end{remark}

We need the following technical result. 

\begin{lemma}\label{lem:free}
    The action of $\Delta_n$ on $\PP^1_q$ is free and transitive. 
\end{lemma}
\begin{proof}
    It is sufficient to show that the stabilizer of $[0:1]\in \PP^1_q$ in $\Delta_n$ is trivial. A matrix in $PSL(2,q)$ is in the stabilizer if and only if it is lower triangular. Every matrix in $\Delta_n$ is of the form $[x,y;-y,x]$ or $[x,y;y,-x]$ and it is lower triangular if and only if it is equal to $[1,0;0,1]=I_2$ or $[1,0;0,-1]$. The latter case has determinant $-1$ which is not a quadratic residue, hence impossible.  
\end{proof}
We can now state and prove the following projectivity result. Note that, provided that we have proved that this is a regular dihedral frame, the projectivity already follows from the previous section. Theorem \ref{thm:main} implies that the Paley ETF is regular, as it is concocted from a dihedral skew Hadamard matrix. Then Corollary \ref{cor:regular}(b) shows that it is genuine projective. Nevertheless, we give now a specialized proof.

\begin{thm}\label{thm:paley}
    When $q$ is a prime power and $q\equiv 3\mod 4$, the Paley $\etf{q+1}{(q+1)/2}$ is a genuine projective dihedral frame. 
\end{thm}

\begin{proof}
    The automorphism group of the Gram matrix $F:=I+\tfrac{\i}{\sqrt{2n-1}}(H_q-I)$ contains a copy of the dihedral group. If $\Phi^*\Phi=F$ then $\Phi$ must be reduced, since its off-diagonal entries have different moduli than its diagonal ones. By the freeness of the action of $\Delta_n$, Lemma \ref{lem:red_aut} implies that $\Phi$ is switching equivalent to some $\overline{D_nv}$, $n=(q+1)/2$, via the representation $\Pi|_{\Delta_n}$. It remains to show that it is a genuine projective frame. We will show that the representation $\Pi|_{\Delta_n}:\Delta_n\to U(2n)$ of Lemma \ref{lem:dihedprojrep} is not equivalent to a strict representation. We are allowed to modify the choice of the representatives $\hat P\in \PP^1_q$. Consider the point $P_0=[0:1]\in \PP^1_q$ and let $\widehat{P_0}=[0,1]$. By Lemma  \ref{lem:free} the set $\{m^iP_0,m^itP_0\}_{1\le i\le n}=\PP^1_q$, and we may define $\widehat{m^iP_0}:=\mathbf m^i[0,1]^\top$ and $\widehat{m^itP_0}:=\mathbf{m}^i\mathbf{t}[0,1]^\top$. This modifies $\Pi$ by phase monomial equivalence. Let $M_2:=\Pi(m)$ and $T_2:=\i\Pi(t)$. By the definition of $\Pi$ and our choice of the representatives, we have $M_2^n=-I_n$, $T_2^2=I_n$ and $T_2M_2T_2^{-1}=M_2^{-1}$, which is exactly the prescription to be a projective dihedral representation, as per Lemma \ref{lem:proj}.
\end{proof}

\begin{remark*}
    Theorem \ref{thm:paley} and Lemma \ref{lem:free}  imply that the automorphism group of the Paley ETF contains a free transitive copy of the dihedral group. The full automorphism group was computed by Iverson and Mixon \cite{IverMix2024}.
\end{remark*}

\subsection{Double Paley ETFs}
In this subsection we continue to assume that $q$ is a prime power congruent to $3$ modulo 4.
Using the doubling construction \eqref{equ:double}, the matrix $$DH_q=\begin{bmatrix}
        H_q & H_q\\ -H_q^\top & H_q^\top
    \end{bmatrix}$$ is a skew Hadamard matrix, leading to an $\etf{2q+2}{q+1}$ constructed by equation \eqref{eq:shm2etf}, which we term as \emph{double Paley} ETF. Our goal is to show that this configuration is again genuine projective dihedral. Recall the projective representation $\Pi:PGL(2,q)\to U(q+1)$. We will construct a new representation $D\Pi:PGL(2,q)\to U(2q+2)$ as follows: 
    \be \label{eq:dpi} D\Pi(g) =\begin{cases}
        [\Pi(g),0;0,\Pi(g)] & g\in PSL(2,q)\\
        [0,-\Pi(g);\Pi(g),0] & g\in PGL(2,q)\setminus PSL(2,q)
    \end{cases}.\ee

It is readily seen that $D\Pi$ is a projective representation. This is a general construction which is in fact the tensor product of $\Pi$ with a two-dimensional projective representation of $PGL(2,q)/PSL(2,q)\cong \ZZ/2\ZZ$. Moreover, $\overline{D\Pi}:PGL(2,q)\to U(2q+2)/\{\pm I\}$ is a group homomorphism. We claim
\begin{lemma}\label{lem:dpirep}
    For all $g\in PGL(2,q)$, $D\Pi(g)\cdot DH_q\cdot D\Pi(g)^*=DH_q$.
\end{lemma}

\begin{proof}
    Let $\rho\in PGL(2,q)\setminus PSL(2,q)$  be the matrix $\diag(-1,1)$, and set $S=\Pi(\rho)$. Then by Lemma \ref{lem:dihedprojrep}(c) $SH_qS^*=H_q^\top$ and $DH_q=[H_q,H_q;-SH_qS^*,SH_qS^*]$. Let $g\in PGL(2,q)$ be any element. If $g\in PSL(2,q)$, then Lemma \ref{lem:dihedprojrep} implies that $\Pi(g)H_q\Pi(g)^*=H_q$ and $\Pi(g)SH_qS^*\Pi(g)^*=\Pi(g\rho)H_q\Pi(g\rho)^*=H_q^\top=SH_qS^\top$. This shows that $D\Pi(g)\cdot DH_q \cdot D\Pi(g)^*=DH_q$. To prove the claim for $g\notin PSL(2,q)$, it suffices to take a specific element, say $g=\rho$. Then we use $S(\rho)^2=\pm I$ to check directly that $D\Pi(\rho)\cdot DH_q\cdot D\Pi(\rho)^*=DH_q$. The details are left to the reader.
\end{proof}

To establish the dihedrality of $DH_q$ we need to find a copy of $D_{2n}$ in $PGL(2,q)$, $n=(q+1)/2$. We use again the embedding of $\FF_{q^2}^\times$ in $GL(2,q)$, but this time we do not need to worry about the determinant. Let $g=a+b\sqrt{-1}$ be a generator of this group, which is an element of order $q^2-1$. It corresponds, as discussed above, to the matrix $[a,-b;b,a]$, but this time its image in $\FF_{q^2}^\times/\FF_q^\times$ is of order $q+1$, and accordingly the element $m_1\in PGL(2,q)$ corresponding to that matrix has order $q+1$. This identifies the cyclic part of $D_{2n}$. Let $t_1=t$ be the same matrix from Lemma \ref{lem:mt}. Then as above $m_1,t_1$ satisfy the relations of the dihedral group, and yield a copy $\Delta_{2n}:=\langle m_1,t_1\rangle\subset PGL(2,q)$ of $D_{2n}$. Similar to Paley ETFs, we have the following.
\begin{thm}\label{thm:paley_d}
    The double Paley ETF is a genuine projective dihedral frame through the projective representation $D\Pi|_{\Delta_{2n}}$.
\end{thm}

\begin{proof}
By choosing $m=m_1^2$ and $t=t_1$, the group $\Delta_n:=\langle m,t \rangle\subset PSL(2,q)$ is a copy of $D_n$ as discussed in Lemma \ref{lem:mt} and in addition $\Delta_n< \Delta_{2n}$ with $[\Delta_{2n}:\Delta_n]=2$. By Lemma \ref{lem:dpirep} $\Delta_{2n}$ acts as automorphisms on $DH_q$. Equation \eqref{eq:dpi} together with Lemma \ref{lem:free} assures that the underlying permutation group of  $D\Pi(\Delta_{2n})$ acts transitively on the standard basis of $\CC^{4n}$, and by the cardinality of $D_{2n}$ the action is free. Therefore, the double Paley ETF is dihedral. By \eqref{eq:dpi} $D\Pi|_{\Delta_n}=\Pi|_{\Delta_n}\oplus \Pi|_{\Delta_n}$, hence the Schur multiplier of $D\Pi|_{\Delta_{n}}$ reduces via restriction to the top left half block to the one of $\Pi$ on $\Delta_n$, which is nontrivial. This proves $D\Pi$ is genuinely projective.
\end{proof}

\begin{remark}
    In general, doubling of a 2-circulant configuration need not be 2-circulant. The same goes for dihedral. In \cite[Thm 28]{IJM24} the Paley and double Paley ETFs appear even for $q\equiv 1\mod 4$, but they are only 2-circulant.
\end{remark}

\section{Numerical and exact results}\label{sec6}

In this section, we present some numerical and exact results pertaining to the search for  regular dihedral $\etfn$. 

\subsection{Numerical Results}
In the search for dihedral ETFs we use the following procedure. We begin with the projective representation of $D_n$ sending $\mu^i\tau\to \hat M^i\hat T$, the matrices that are defined above in equation \eqref{equ:MT2}. For each complex vector $v\neq 0$ we form the frame $\Phi=D_nv$, and substitute its vectors into the frame potential function:
$$ F_p(\Phi)\ := \ \sum_{i,j} |\langle \varphi_i,\varphi_j\rangle|^p,$$ for some values of $p$, namely $p=3,4,5,6$. Finally, the Matlab function @fminsearch is called on the function $v\mapsto F_p(D_nv)$ and returns an approximate minimizer $\Phi_{approx}$. This minimum is regular projective dihedral by construction, and we declare success if all angles $|\langle \varphi_i,\varphi_j\rangle|$ are equal up to a relative error of $10^{-7}$ or less.

The numerical search was initially performed on the strict representation as in equation \eqref{equ:MT}, but there were no solutions as we now expect from Theorem \ref{thm:2negcirc}. On the other hand, in the projective case, we have found solutions for $n=2,4,6,8,10,12,14$. There were also fake solutions for $n=16,18,20$, but these had a much larger error term: The angles differed by more than 0.1\%. It is possible to exactify the solutions, so to obtain an exact ETF. This is because of Theorem \ref{thm:2negcirc}, which tells us that any regular solution is derived from a \bnc\ skew Hadamard matrix $H$. By using this recipe to extract $H$ from an approximate solution $\Phi_{approx}$ we obtain a real matrix $H_{approx}$ with values that are near $\pm 1$. We can round the entries to the nearest integer and check whether we got a Hadamard matrix (the skewness and negacircularity are inherent from the structure). For $n\le 14$ this has been the case. For $n>14$ we did not get a Hadamard matrix. We have tried to modify a small number of entries in the first row and accordingly in the rest of the matrix to see whether we are off by a small perturbation. This did not rectify the situation, which leads us to think that Matlab has found local minima that are far away from an absolute minimizer. We do know that at size $n=16,20$ there are dihedral ETFs, for instance those coming from the double Paley construction.

The next step was to compare the results to the Paley and double Paley ETFs. More precisely, we tried to see whether an ETF is switching equivalent to a Paley or a double Paley matrix. It follows from Lemma \ref{lem:dihedprojrep}(c) that the Paley ETF is switching equivalent to its complex conjugate. However, double Paley ETFs are in general not. It turned out that all of the examples we have found by numerical methods were switching equivalent to exactly one of Paley, double Paley, or the conjugate double Paley ETF, which we term collectively as \emph{Paley type}. In the first place, the numerical results gave us the clue to the existence of the double Paley family, which we were not aware of.

\subsection{Exact Solutions} In addition, we have searched for exact regular projective dihedral ETFs by searching for \bnc\ skew Hadamard matrices. In view of Theorem \ref{thm:main} we only need to search for a pair $(P,Q)$ of negacirculant matrices with $P$ being skew, satisfying $PP^\top+QQ^\top=2nI$. Moreover, $P$ and $Q$ are determined from their first row and replacing the first row of $Q$ by a nega-rotation: $[v_1,\ldots,v_n]\mapsto [-v_n,v_1,\ldots,v_{n-1}]$, results in a switching equivalent matrix. All these are used to reduce the search space for regular dihedral ETFs.
We have conducted an exhaustive search up to $n=22$, and a sample search for $n=24,26,28$. We have found solutions in all cases, except for $n=18$. This can be thus stated as a theorem:
\begin{thm}
    There is no regular dihedral $\etf{36}{18}$.
\end{thm}
Then we have checked all of our solutions against the three Paley types. It turns out that 

\begin{thm}\label{thm:PaleyType}
    All regular dihedral $\etfn$ for $n\le 22$, $n\neq 16$ are of Paley type. There exist regular dihedral $\etfn$ not of Paley type for $n=16,24,26$.
\end{thm}
The proof will be given below. All of the examples we have found for $n=28$ are of Paley type. 

\subsection{Verifying whether two dihedral ETFs are switching equivalent.} \label{sec:ver_isom}
Let $\Phi_0,\Phi_1$ be two dihedral $\etfn$, that we wish to test whether they are switching equivalent. For the discussion ahead, we will only need to assume that they are both transitive under their automorphism groups. Suppose that they are given by their Grams $G_0,G_1$. We need to find an isomorphism between $G_0$ and $G_1$ or show that it does not exist. By an isomorphism we mean a complex monomial matrix $L$ such that $LG_0L^*=G_1$. We apply complex diagonal transformations to each $G_0,G_1$, after which we may assume that the first row and column of both $G_0,G_1$ are positive real. Since the automorphism groups of $G_0$ and $G_1$ are transitive, we may further assume without loss of generality that the underlying permutation of $L$ fixes the first coordinate. Moreover, the normalization of the first rows/columns  forces $L$ to be a permutation matrix with $L_{0,0}=1$.
Let $G_i'$ be the matrix obtained by $G_i$ after removing the first row and column.
Treating $G_i'$ as weighted directed graphs, our problem reduces to the (weighted) graph isomorphism problem.

Our next step is to use Sagemath \cite{sage}, which exactly provides this functionality, based on an algorithm of McKay \cite{mckay1981practical}. If a permutation matrix exists, Sagemath spits one out and we can readily verify the switching equivalence. Otherwise, as we do not want to count on Sagemath, we proceed as follows. To rule out the existence of a permutation, we need to loop over $i=1,2,\ldots, 2n-1$, assuming that the first row of $G_0'$ is mapped under the potential permutation to the $i$th row of $G_1'$. WLOG, let us rule out the case $i=1$. We partition the matrices $G_i'$ according to the values in the first row. They must have the same multiset of values. If not, we are done. Otherwise, we partition the index set of $G_0'$ according to the value of $(G_0')_{1,j}$. Let the partition be $I_1\cup I_2\cup\cdots \cup I_m$, and similarly we do for $G_1'$, to obtain a partition $J_1\cup J_2\cup \cdots\cup J_m$. We do this so that $(G_0')_{1,j}=(G_1')_{1,j'}$ if and only if $j\in I_r$ and $j'\in J_r$ for some $r$. Now it is clear that if a permutation exists, it must map each submatrix $A_r:=(G_0')_{I_r\times I_r}$ to $B_r:=(G'_1)_{J_r\times J_r}$. In particular, $A_r,B_r$ must have the same characteristic polynomial. If not for some $r$, then we are done.

In practice, this was all we needed to prove the inequivalence of a certain ETF to a Paley type. In fact, we did not need to loop over all $i$, but just over two values of $i$, as the set of all pairs of indices $(i,j)$ with $i\neq j$ has at most two orbits under the automorphism group of the Paley type matrix. This gives a complete proof of Theorem \ref{thm:PaleyType}.

\subsection{Classification of regular dihedral $\etfn$ for $n\le 22$}
Table \ref{tab:exact} gives a full list of all regular dihedral $\etfn$ for $n\le 22$ up to switching equivalence, and the Paley type if it is of that type. No two members in the list are isomorphic. Each entry is given as a pair of two vectors, $a,b$. For the Paley type we write P, DP, CDP if it is of Paley, double Paley or conjugate double Paley respectively. The way to recover the ETF from $a,b$ is to set up the negacirculant matrices $A=\negc(a), B=\negc(b)$, then set up the skew Hadamard matrix $H=[A,B;-B^\top,A^\top]$ and finally compute the Gram matrix as is given by Theorem \ref{thm:2negcirc}. The classification is a consequence of the technique and proof explained in \S\ref{sec:ver_isom}. We have also produced partial database for $n=24,26,28$ in Table \ref{tab:partial}.

\begin{table}[hbt]
\caption{Complete classification of regular dihedral $\etfn$ for $n\leq 22$}\label{tab:exact}
\begin{tabular}{c|lll}
$n$ &$a$& $b$ & Symmetry Type\\\hline\hline
2&$(1, -1)$& $(-1, -1)$ & P\\\hline
4& $(1, -1, -1, -1)$ & $(-1, -1, 1, -1)$ & $\text{P}\cong \text{DP}\cong \text{CDP}$\\\hline
6& 24 & 02 & P\\\hline
8&F7 & ED & DP\\
&F7& E9 & CDP\\\hline
10 & 3EF& 353 & P\\\hline
12 &F77 &F4D & P\\
& E8B & F7B & CDP\\
& E8B & F79 & DP\\\hline
14 & 3953&3EDF&P\\\hline
16 & F227&FD65 & P\\
& F227&FBAD & --\\
& F227 & FA51 &--\\\hline
20 & FB76F & EE2D7 & DP\\
	&FB76F & EE297 & CDP\\\hline
22 &3D9537& 3FD38D & P\\\hline
\end{tabular}
\end{table}

The entries of the vectors $a, b$ are 1 or -1 (as seen in the first two rows of Table \ref{tab:exact}), which is binary, so we express them in  Hex form when $n\geq6$ in Table \ref{tab:exact} and Table \ref{tab:partial}.
For example, when $n=6$, `24' in hex converted to binary is `100100', and therefore the vector $a$ is (1, -1, -1, 1, -1, -1).
\begin{table}[hbt]
\caption{Partial classification of dihedral $\etfn$}\label{tab:partial}
\begin{tabular}{c|lll}
$n$ &$a$& $b$ & Symmetry Type\\\hline\hline
24 &D180C5& F84D59 & DP\\
	&C5F7D1&533CF4&--\\
	&943614&0B2211& --\\\hline
26 &3257D49& 1881C3B & --\\\hline
28 & 8298CA0 & A5064C1 & DP\\\hline
\end{tabular}
\end{table}

For $n=24$ our list misses the Paley and the conjugate double Paley, which implies that there are at least five isomorphism classes. For $n=26$ there are no Paley type available, but there are still a solutions. 
For $n=28$ our list misses the conjugate double Paley, so there are at least two isomorphism classes.

\section*{Acknowledgments}
X. Chen was partially supported by the National Science Foundation, grant DMS-2307827.
K.~A.~Okoudjou was partially supported by the National Science Foundation, grants DMS-2205771 and DMS-2309652.

We would like to thank an anonymous reviewer for a very careful and thorough scan of the manuscript, and for suggesting many fixes and tips. We would like to thank another anonymous reviewer for spotting errors and bugs in the original manuscript, and suggesting good fixes for most of them. Both reviewers did an excellent and professional work.

\bibliographystyle{amsplain}
\bibliography{Dihedral} 

\end{document}